\documentclass[preprint,12pt]{elsarticle}

\usepackage{mypackages,mymathmacros,mynotation,mytikzconfig}
\usepackage[nameinlink,capitalize]{cleveref}

\journal{Mechanical Systems and Signal Processing}

\begin{document}

\begin{frontmatter}



\title{Adaptive Reduced Order Modelling of Discrete-Time Systems with Input-Output Dead Time}


\author[1]{Art J. R. Pelling} %% Author name
\author[1]{Ennes Sarradj}

\affiliation[1]{organization={Department of Engineering Acoustics, Technische Universtiät Berlin},%Department and Organization
            addressline={Einsteinufer 25}, 
            city={Berlin},
            postcode={10587}, 
            state={Berlin},
            country={Germany}}

%% Abstract
\begin{abstract}
While many acoustic systems are well-modelled by linear time-invariant dynamical systems, high-fidelity models often become computationally expensive due the complexity of dynamics. Reduced order modelling techniques, such as the Eigensystem Realization Algorithm (ERA), can be used to create efficient surrogate models from measurement data, particularly impulse responses. However, practical challenges remain, including the presence of input-output dead times, i.e. propagation delays, in the data, which can increase model order and introduce artifacts like pre-ringing. This paper introduces an improved technique for the extraction of dead times, by formulating a linear program to separate input and output dead times from the data. Additionally, the paper presents an adaptive randomized ERA pipeline that leverages recent advances in numerical linear algebra to reduce computational complexity and enabling scalable model reduction. Benchmarking on large-scale datasets of measured room impulse responses demonstrates that the propsed dead time extraction scheme yields more accurate and efficient reduced order models compared to previous approaches. The implementation is made available as open-source Python code, facilitating reproducibility and further research.
\end{abstract}


\begin{highlights}
    \item Randomized ERA is combined with a recent randomized error estimator that allows adaptive identification of reduced order state space models.
    \item Improved dead time extraction for MIMO systems via linear program.
    \item Largest application of ERA to date in terms of input data dimension and order of the constructed models.
    \item Correction of the classical ERA error bound due to a typing error in the original work.
\end{highlights}

%% Keywords
\begin{keyword}
acoustics \sep dynamical systems \sep model order reduction \sep randomized \sep Eigensystem Realization Algorithm \sep room impulse response


\MSC[2020] 93B15 \sep 93B30 \sep 93A15 \sep 93C55 \sep 68W20
\end{keyword}

\end{frontmatter}


\section{Introduction}
Many acoustic systems can be modelled as \ac{lti} dynamical systems. \ac{lti} systems possess several favourable properties and are well-studied. However, for practically relevant acoustical engineering tasks, one oftentimes finds that \ac{lti} models can become large because many acoustic systems, although being linear, encode complex wave phenomena. It can be necessary to employ more efficient surrogate models when these models are part of a larger ensemble of multi-physics simulations or need to be evaluated perpetually. For this reason, the discipline of \ac{mor} \cite{antoulas2005,benner2017} has gained traction in acoustical engineering recently. \ac{mor} methods provide ways to represent complex dynamical systems with fewer parameters while maintaining accuracy. Applications of \ac{mor} can be found in acoustics whenever high-fidelity models become computationally expensive or impractical, e.g. in vibro-acoustics \cite{aumann2019,aumann2021,aumann2022}, for modelling of musical instruments \cite{
maestre2021,rettberg2023}, in thermoacoustics \cite{emmert2016,emmert2016a,brokof2023,meindl2016}, and in ocean acoustics \cite{candy2020,candy2021,candy2021a}. We refer to \cite{deckers2021} for an overview of \ac{mor} methods in acoustics and to the references therein for further applications.

In the field of \ac{mor}, \ac{lti} models in so-called state space form are predominantly considered over more classical \ac{ir} or transfer function models. As argued in \cite{pelling2021}, state space models can be superior, especially when considering \ac{mimo} systems. In the latter architectures, a \ac{mimo} system is modelled as a collection of independent \ac{siso} systems, which means that the model size and computational complexity scales with the number of inputs times the number of outputs, as is illustrated in \cref{fig:naive}. With a state space architecture, on the other hand, redundancies in the in- and outputs are avoided. The model complexity is determined mostly by the internal dynamics of the system, denoted by \(\mathcal{S}\) in \cref{fig:ssm}.
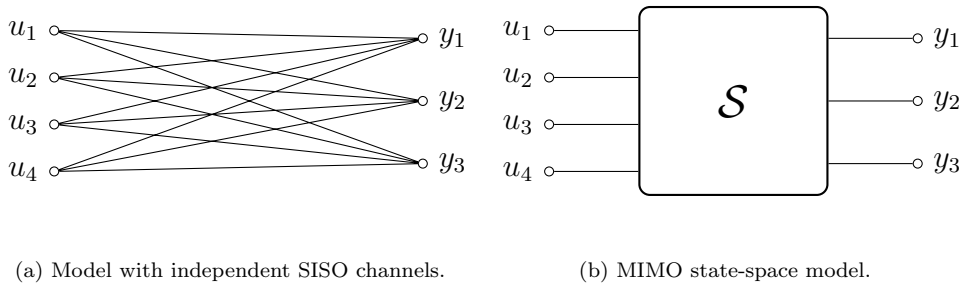
\begin{figure}[h]
    \centering
    \begin{subfigure}{.48\textwidth}
    \centering    
    \pgfmathsetmacro{\m}{4}
\pgfmathsetmacro{\p}{3}
\begin{tikzpicture}
\node at (-2,0) {};
\node at (2+\bsize,\bsize) {};
\draw
\foreach \i in {1,...,\m} {
    node (in\i) at (-1.2cm,\pos{\i}{\m}) [draw,circle,label=left:\(u_{\i}\),inner sep=1.2pt] {}
};
\draw
\foreach \j in {1,...,\p} {
    node (out\j) at (1.2cm+\bsize,\pos{\j}{\p}) [draw,circle,label=right:\(y_{\j}\),inner sep=1.2pt] {}
};
\draw
\foreach \i in {1,...,\m} {
\foreach \j in {1,...,\p} {
(in\i) -- (out\j)
}};
\end{tikzpicture}
    \caption{Model with independent \ac{siso} channels.}
    \label{fig:naive}    
    \end{subfigure}%
    \begin{subfigure}{.48\textwidth}
    \centering    
    \pgfmathsetmacro{\m}{4}
\pgfmathsetmacro{\p}{3}
\begin{tikzpicture}
\node at (-2,0) {};
\node at (2+\bsize,\bsize) {};
\draw
node [myblock] (sys) at (\bsize/2,\bsize/2) {\Large\(\mathcal{S}\)};
\draw
\foreach \i in {1,...,\m} {
    node (in\i) at (-1.2cm,\pos{\i}{\m}) [draw,circle,label=left:\(u_{\i}\),inner sep=1.2pt] {}
    (in\i) -- ([yshift=(\pos{\i}{\m}-\bsize/2)]sys.west)
};
\draw
\foreach \j in {1,...,\p} {
    node (out\j) at (1.2cm+\bsize,\pos{\j}{\p}) [draw,circle,label=right:\(y_{\j}\),inner sep=1.2pt] {}
    (out\j) -- ([yshift=(\pos{\j}{\p}-\bsize/2)]sys.east)
};
\end{tikzpicture}       
    \caption{\ac{mimo} state-space model.}    
    \label{fig:ssm}    
    \end{subfigure}
    \caption{Schematic representation of different \ac{lti} model structures for a \ac{mimo} system with four inputs \(u_1,u_2,u_3,u_4\) and three outputs \(y_1,y_2,y_3\).}
    \label{fig:model-comparison}
\end{figure}

The above-mentioned applications of \ac{mor} in acoustics deal with the reduction process of a so-called \ac{fom} in state space form that has been obtained either by first order principles or the discretization of a partial differential equation, e.g. by the finite element method. The reduction process usually entails some form of projection of the \acp{dof} of the \ac{fom} onto a smaller subspace, by which a \ac{rom} is obtained. In this work, we will consider a scenario where a \ac{fom} is unavailable and, instead, one has to rely on measurement data --- a scenario that is often encountered in acoustical engineering because many model properties needed for a constructive modelling approach are not determined. For example, material parameters like angle-dependent absorption and scattering coefficients or (boundary) domain geometries might not be known to a required accuracy that can ensure a high-fidelity model. Precisely because constructive modelling methods can be highly sensitive to these input properties, experimental measurements of real-world acoustic systems are common practice in the field and, in turn, measurement data is abundantly available.

The problem of obtaining a \ac{rom} not by the reduction of a \ac{fom} but from measurement data is known as data-driven \ac{mor} or \emph{reduced order modelling}, as a \ac{rom} is obtained directly from the data without constructing a \ac{fom} from the data first. An established way of experimental characterization of dynamical systems in acoustics is the measurement of \acp{ir}, making this type of data ubiquitous. The so-called \ac{era} or Ho-Kalman algorithm \cite{juang1985,ho1966a} is a reduced order modelling method that has been successfully employed for acoustical modelling and requires input data in the form of \acp{ir}. Thus, \ac{era} is a natural choice and will be considered in this work. \cref{fig:d1-tf} offers a qualitative impression of \acp{rom} that can be obtained from measured \acp{ir} of acoustic systems with \ac{era}. The plot depicts a single channel frequency response of different \acp{rom}, which are actually \ac{mimo} systems (as schematically suggested in \cref{fig:ssm}). A detailed description of the considered \ac{ir} measurement data and constructed \acp{rom} is provided later in this document.
\begin{figure}
\centering
\begin{subfigure}{.5\textwidth}
    %\captionsetup{aboveskip=-1em}
    \centering
    \includegraphics{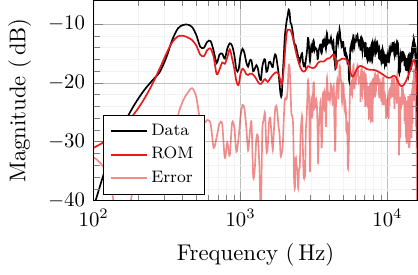}
    \caption{\ac{rom} of order \(380\).}
\end{subfigure}%
\begin{subfigure}{.5\textwidth}
    %\captionsetup{aboveskip=-1em}
    \hfill
    \includegraphics{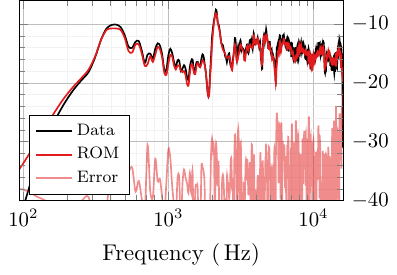}
    \caption{\ac{rom} of order \(3330\).}
\end{subfigure}
\caption{Frequency responses of different \acp{rom} for the \acs{miracle}-D1 scenario. Each subplot depicts the frequency response of the \ac{rom} alongside the magnitude of the Fourier-transformed \ac{ir} measurement and the resulting modelling error. The \acp{rom} are constructed with \ac{era} according to \cref{sec:results}. Each \ac{rom} is actually a \ac{mimo} system with \(1024\) inputs and \(64\) outputs; the subplots only contain the transmission between the first input and first output.\label{fig:d1-tf}}
\end{figure}

The first applications of \ac{era} to acoustic systems identify  \acp{rom} for the control of an acoustic duct \cite{hong1996} and the modelling of aerodynamic forces on a plate \cite{frampton1996}, respectively. Notably, \ac{era} has been employed in fluid dynamics \cite{brunton2014,ma2011,kramer2016}, structural vibration \cite{caicedo2011,lam2011,wang2020a}, and aero-acoustic propagation modelling \cite{ketcham2012}. Further, it has found recent use for virtual acoustic rendering in the form of the reduced order modelling of \acp{hrtf} \cite{adams2008,georgiou1999,georgiou2000,mackenzie1997} and \acp{rir} \cite{pelling2021,kujawski2024,hilgemann2023}\footnote{Although slightly different, the method in \cite{hilgemann2023} can be considered equivalent to \ac{era}.}.

The majority of aforementioned applications of \ac{era} in acoustics consider medium-sized systems with short \acp{ir} and a few in- and outputs, as the classical \ac{era} approach becomes infeasible for high-dimensional measurement data. Its main computational bottleneck is posed by a \ac{svd} of an associated Hankel matrix. The computational complexity of \ac{svd} scales cubically with the data dimension, i.e. the number of samples that the measured \acp{ir} contain. A few straightforward computational improvements are offered in \cite{ketcham2012,hilgemann2023} that avoid unnecessary computations by considering the rank of the Hankel matrix. While being more efficient, these approaches do not solve the fundamental problem of scalability. As \ac{ir} measurements of acoustical systems can regularly contain millions of samples, the applicability of \ac{era} in acoustical engineering was limited. Recently, it has been demonstrated in \cite{pelling2021,kujawski2024} that \ac{era} can be applied to substantially larger datasets when leveraging recent randomized matrix approximation methods \cite{minster2021,halko2011,martinsson2020}. This fundamental shift in computational efficiency makes \ac{era} a feasible tool for reduced order modelling in nearly all practical scenarios. Nonetheless, to turn randomized \ac{era} into an adept reduced order modelling tool for engineering practice, several challenges remain:

Firstly, due to source-receiver distances, acoustic \ac{ir} measurements can contain large propagation delays. An unstructured \ac{rom} may require numerous states to resolve these propagation delays correctly, even if the underlying delay-free dynamics can be well modelled by a low order \ac{rom}. Propagation delays not only increase the length of \ac{ir} measurements but can lead to pre-ringing or \emph{impulse smearing} artifacts in \acp{rom} that are not large enough to fully resolve the propagation delays, as reported in \cite{pelling2021,adams2008}. As a remedy, the prevalent approach so far has been to remove the least-common propagation delay in the measurements before employing \ac{era} \cite{hilgemann2023,pelling2021,adams2008,kujawski2024,georgiou2000} and then reapplying the propagation delays to the identified \ac{rom}. While this approach is simple and practical, oftentimes the propagation delays are not compensated for completely. In this work, we will introduce an improved technique that can compensate for a wider range of propagation delays, leading to improved \acp{rom}.

\begin{remark}
The term \emph{delay} finds multiple use in different scientific areas. Importantly, in the context of dynamical systems and \ac{mor}, it conventionally refers to delay differential (algebraic) equations \cite{ha2016,fridman2014} that possess delays in the state. Therefore, we will use the term \emph{dead time} from now on, a terminology found
in control theory \cite{meinsma2000,normey-rico2007}, whenever we refer to propagation delays in the inputs and outputs to avoid confusion.
\end{remark}

A second issue of practical nature concerns the computational workflow of randomized \ac{era}. As the minimum required order of a \ac{rom} that possesses the desired accuracy is not known a priori, finding a small, yet accurate, \ac{rom} can entail multiple computations of increasing order or computing a large \ac{rom} by vastly overestimating the required order and subsequently reducing the obtained \ac{rom} further with common \ac{mor} methods. In effect, both design approaches can be lengthy as they contain unnecessary or repeated computations. We tackle this problem by reformulating the randomized \ac{era} procedure outlined in \cite{minster2021} as an adaptive pipeline. To this end, we employ an efficient heuristic randomized \ac{loo} error estimator \cite{epperly2024}, a memory efficient and updatable QR algorithm \cite{fukaya2020}, and revisit and rectify the classical error bound of \ac{era} \cite{kung1978}. To further facilitate the use of the proposed adaptive randomized \ac{era} algorithm and the reproducibility of reported results, a computationally scalable and efficient reference implementation in the Python programming language is published alongside this work.

The paper is structured as follows. In \cref{sec:preliminaries}, mathematical preliminaries such as notation, discrete-time state space systems, \ac{era}, and the proposed structured dead time systems are introduced. \cref{sec:theory} is dedicated to the proposed improvements of randomized \ac{era} and begins by formally introducing the \ac{dts} problem in \cref{sec:problem}. After a short overview of existing so-called \ac{tde} methods in \cref{sec:tde}, we present our proposed dead time extraction scheme based on a solution of the \ac{dts} problem in \cref{sec:dts}. According to that, the main algorithm is developed in \cref{sec:numerical-implementation}, after which the proposed error estimator is introduced in \cref{sec:err-est}. \cref{sec:benchmarks} introduces the considered benchmark datasets of measured \acp{rir} and computational resources with which the proposed method is analysed in \cref{sec:results}. The analysis contains an investigation of the different existing and introduced error bounds and estimators in \cref{sec:error-bounds-and-estimators}, demonstrates the superiority of the proposed dead time extraction scheme in \cref{sec:dte}, and closes with an assessment of overall performance. Finally, a conclusion is offered in \cref{sec:conclusion}.
\section{Preliminaries}
\label{sec:preliminaries}
\subsection{Notation}
Throughout this document, matrices and vectors are denoted in boldface whereas scalar quantities are not, i.e. \(A\in\C[m][n]\), \(\bx\in\C[n]\) and \(a,N\in\C\), respectively. \(\norm[\mathrm{F}]{A}\) denotes the
Frobenius norm, \(A^{\dagger}\) denotes the Moore-Penrose pseudoinverse and
\(\sigma_{i}(A)\) denotes the \(i\)-th singular value of \(A\). The Hilbert space of
square-summable sequences is denoted by \(\ell_{2}\). The value of any sequence \(a\in\ell_2\) at time index \(t\) is denoted by \(a(t)\).

\subsection{Discrete-time \acs{lti} systems}
A discrete-time system \(\mathcal{S}\) is a mapping from an input \(\bu\in\ell_2^m\) to an output \(\mathcal{S}(\bu)=\by\in\ell_2^p\). It is said to be time-invariant if it does not depend on absolute time. Further, it is called linear, if \(\mathcal{S}(\alpha\bu_1+\beta\bu_2)=\alpha\mathcal{S}(\bu_1)+\beta\mathcal{S}(\bu_2)\).
The dynamical action of a \ac{mimo} discrete-time causal \ac{lti} system with \(m\in\N\) inputs and \(p\in\N\) outputs is fully characterized in the time domain by a convolution
\begin{equation}
\label{eq:convolution}
    \by(k)=\sum_{k=0}^\infty h(t-k)\bu(t),
\end{equation}
where \(h\in\ell_2^{p\times m}\) denotes the system's \ac{ir}. Equivalently, in the frequency domain (or \(\mathcal{Z}\) domain) the transfer function \(\sys(z)\) describes the system dynamics via
\begin{equation*}
    Y(z)=\sys(z)U(z),
\end{equation*}
where \(U(z)\) and \(Y(z)\) denote \(\mathcal{Z}\)-transforms of the input and output, respectively.
Let \(\mathbb{D}=\{z\in\C:|z|<1\}\) denote the open unit disk. The space of all holomorphic functions \(\sys:\mathbb{D}\rightarrow\C^{p\times m}\) satisfying
\begin{equation*}
    \norm[\cH_2^{p\times m}(\mathbb{D})]{\sys}=\biggl(\,\frac{1}{2\pi}\int_{-\pi}^{\pi}\norm[\mathrm{F}]{\sys(e^{\imath\omega})}^2\mathrm{d}\omega\biggr)^{1/2}< \infty
\end{equation*}
forms a so-called Hardy space. By Parseval's theorem, it holds
\begin{equation}
\label{eq:parseval}
    \norm[\cH_2^{p\times m}(\mathbb{D})]{\sys}=\biggl(\,\sum_{k=0}^\infty\norm[\mathrm{F}]{h_k}^2\biggr)^{1/2}.
\end{equation}
The space \(\cH_2(\mathbb{D})^{p\times m}\) will be referred to as \(\cH_2\) for short in the following.

\subsection{Discrete-time state-space models}
\label{sec:dt-ssm}
A more general system formulation can be obtained when augmenting the model by adding an internal state \(\bx\in\ell_2^n\) instead of merely considering the direct mapping from input to output. In this case, the following so-called state equations describe the system dynamics
\begin{equation}
\label{eq:state}
\begin{aligned}
    \bx(t+1)&=A\bx(t)+B\bu(t),\\
    \by(t)&=C\bx(t)+D\bu(t),
\end{aligned}
\end{equation}
where \(A\in\R[n][n]\) represents the internal dynamics, \(B\in\R[n][m]\) is the input-to-state mapping, \(C\in\R[p][n]\) is the state-to-output mapping and \(D\in\R[p][m]\) is the so-called feedthrough.
The state dimension or system order is denoted by \(n\in\N\). 

Under the assumption of zero initial conditions, i.e. \(\bx(0)=0\), a state-space formulation of the transfer function can be obtained as 
\begin{equation*}
    G(z)=C(zI-A)^{-1}B+D
\end{equation*}
by applying a \(\mathcal{Z}\)-transform to \cref{eq:state} and subsequently solving for and substituting \(\bx\). The moments of the transfer function for expansion at infinity
\begin{equation}
\label{eq:mp}
    h_k=\left.\frac{\mathrm{d}^k}{\mathrm{d} z^k}\sys(z)\right|_{z=\infty}=\begin{cases}
        D,&k=0,\\CA^{k-1}B,&k>0,
    \end{cases}
\end{equation}
are called the Markov parameters of \(\sys\). In discrete-time, they are equal to the \ac{ir} of the system. A matrix quadruple \((A,B,C,D)\), that fulfils \cref{eq:mp} is called a realization of the sequence of Markov parameters \(h\in\ell_2^{p\times m}\).

Throughout this document, we will use \(\sys\) to refer to the \ac{lti} system, its realization \(\sys=(A,B,C,D)\), and its so-called system matrix 
\begin{equation*}\sys=
    \begin{array}{c|c}
        A&B\\\hline C&D
    \end{array}\in\R[(n+p)][(n+m)]
\end{equation*}
interchangeably. Further, we denote the system transfer function by \(\sys(z)\); always with its complex argument to avoid confusion. The transfer function \(G(z)\in\htwo\) is a rational function. The degree of \(G(z)\), i.e., the maximum of the degree of the numerator and the degree of the denominator, is often times referred to as the McMillan degree of the function. If the dimension of \(A\) is equal to the McMillan degree, the system is called minimal.
\subsection{Eigensystem Realization Algorithm}
The \ac{era} goes back to \cite{ho1966,ho1966a} and was later improved in \cite{silverman1971,kung1978,juang1985}, amongst others. For a detailed historical overview of the algorithm, we refer the reader to \cite{schutter2000}. We will now state the version of the algorithm outlined in \cite{kramer2016} together with some more recent results below.

Given a finite sequence \(h\in\ell_{2}^{p\times m}\) of \(2s-1\), \(s\in\N\) Markov parameters, a matching state-space model can be identified via the Hankel matrix of Markov parameters
\begin{equation}
\label{eq:hankel-matrix}
  \cH=
  \begin{bmatrix}
      h_1&h_2&\cdots&h_s\\
      h_2&h_3&\cdots&h_{s+1}\\
      \vdots&\vdots&\iddots&\vdots\\
      h_s&h_{s+1}&\cdots&h_{2s-1}
  \end{bmatrix}=
  \begin{bmatrix}
    CB & CAB & \cdots & CA^{s-1}B\\
    CAB & CA^2B &\cdots & CA^sB \\
    \vdots & \vdots & \iddots &\vdots\\
    CA^{s-1}B & CA^{s}B & \cdots & CA^{2s-2}B
  \end{bmatrix}.
\end{equation}
The Hankel matrix can be factored into the observability and controllability
matrix \(\oMat\in\R[ps][n]\) and \(\cMat\in\R[n][ms]\)
\begin{align*}
  \cH=\underbrace{\begin{bmatrix} C \\ \vdots \\ CA^{s-1} \end{bmatrix}}_{\eqqcolon \oMat}
  \underbrace{\begin{bmatrix} B & \cdots & A^{s-1}B \end{bmatrix}}_{\eqqcolon \cMat}\in\R[ps][ms].
\end{align*}
From this factorization, a realization can be constructed via
\begin{equation}
\renewcommand{\arraystretch}{1}
\label{eq:era-realization}
A=\oMat_{f}^{\dagger}\oMat_{l},\qquad B=\cMat\begin{bmatrix} I_{m} \\ 0 \end{bmatrix},\qquad C=\begin{bmatrix} I_p & 0 \end{bmatrix}\oMat,\qquad D=h_0,
\end{equation}
where $\oMat_{f}$ and $\oMat_{l}$ denote the first and last $p(s-1)$ rows of the observability matrix $\oMat\in\R[ps][n]$, respectively
\cite{kung1978,juang1985}. 

\sloppy In \ac{era}, the factorization is obtained from a (truncated) \ac{svd} \cite{kung1978,zeiger1974}. For a truncation order $r\leq\min\{ms,\,ps\}$, the truncated \ac{svd} is given by
\begin{equation}
\label{eq:tsvd}
  \cH=
  \begin{bmatrix}
    U_{r}&\ast
  \end{bmatrix}
  \begin{bmatrix}
    \bSigma_{r}&0\\0&\ast
  \end{bmatrix}
  \begin{bmatrix}
    V_{r}^{\top}\\\ast
  \end{bmatrix}\approx U_r\bSigma_rV_r^{\top},
\end{equation}
where \(U_{r}\in\R[ps][r]\) comprises the first \(r\) left singular vectors as columns, \(\bSigma_r=\diag{\sigma_1,\,\dots,\,\sigma_r}\) the first \(r\) singular values, and \(V_{r}\in\R[ms][r]\) the first \(r\) right singular vectors as columns. 
By choosing
\begin{equation}
  \label{eq:grams}
  \oMat=U_{r}\bSigma_{r}^{1/2}\quad\text{and}\quad\cMat=\bSigma_{r}^{1/2}V_{r}^{\top},
\end{equation}
as a factorization of the Hankel matrix, one obtains a reduced realization via \cref{eq:era-realization}. 

An important assumption for \ac{era} initially made in \cite{kung1978} is that the Markov parameters decay to zero, i.e.
\begin{equation}
    \label{eq:assumption-kung}
    h_{k}\rightarrow0\quad\text{for}\quad k>s.
\end{equation}
The idea behind it being the following consideration: An asymptotically stable system must have Markov parameters that decay to zero. Thus, the assumption assures that the underlying \ac{fom} is an asymptotically stable system and that the measurements of the Markov parameters contain all relevant dynamics of the system because they are (approximately) zero for any time index \(k>s\). Under these conditions, the \ac{rom} exhibits several desirable properties such as balancedness and stability \cite{akers1995,pernebo1982}, as well as abiding an a priori error bound \cite{kung1978}. The results are summarized in the following theorem:
\begin{theorem}[\cite{akers1995,kung1978,pernebo1982}]
\label{thm:era}
    If the Markov parameters satisfy \cref{eq:assumption-kung}, the realization given by \cref{eq:era-realization} and \cref{eq:grams} is (finitely) balanced, stable and satisfies
\begin{equation}
\label{eq:kung-bound}
\norm[\htwo]{\sys-\sys_r}=\biggl(\,
  \sum_{k=1}^{2s-1}\norm[\mathrm{F}]{\smash{h_k-C_{r}A_{r}^{k-1}B_{r}}}^{2}\biggr)^{1/2}\leq\sqrt{r+m+p}\cdot\sigma_{r+1}(\cH),
\end{equation}
where $\sigma_{r+1}(\cH)$ is the first neglected Hankel singular value.
\end{theorem}
\begin{remark}
    In \cite{kramer2016,kramer2018}, \cref{thm:era} appears in a similar form, but bounding the squared \(\htwo\)-norm on the left-hand side of \cref{eq:kung-bound}. This different version of the error bound is used in the proofs of \cite[Theorem 3.4]{kramer2016} and \cite[Proposition 11]{kramer2018}. Upon investigation of the original work \cite{kung1978}, we believe that the form of the Theorem used in \cite{kramer2016,kramer2018} is flawed, and \cref{eq:kung-bound} is indeed correct. Being written on a typewriter, the original manuscript of Kung \cite{kung1978} contains several typing mistakes and inconsistencies. In fact, Kung introduces a matrix norm \(\norm[\mathrm{M}]{\cdot}=\norm[\mathrm{F}]{\cdot}^2\) in the lead-up to his theorem which is later redefined in the proof in the appendix as \(\norm[\mathrm{M}]{\cdot}=\norm[\mathrm{F}]{\cdot}\). Following along the proof in \cite{kung1978}, it becomes apparent that a power of two is missing in the earlier definition of the norm. Therefore, the above-mentioned derived bounds in \cite{kramer2016,kramer2018} need to be adjusted accordingly. As the outlines of the proofs in \cite{kramer2016,kramer2018} are unaffected by this modification, it can be done easily.
\end{remark}
\subsection{Discrete-time dead time systems}
\label{sec:dead-time}
Many dynamical systems exhibit dead times in their dynamical behaviour. In a \ac{siso} setting, this means that there exists some index \(T>0\) for which \(\by(t)=0\) for \(t<T\) for any \(\bu\in\ell_2\). In other words, it takes a minimal time \(T\) for any action on the input to be observable at the output. By \cref{eq:convolution}, this property is equivalent to \(h(t)=0\) for \(t<T\). The largest such index \(T\) is called the dead time of the system.

\sloppy For \ac{mimo} systems, the definition of dead time is more involved, see \cite{halevi1996, latawiec2000} and \cite[Ch. 11]{normey-rico2007} for details. We now introduce a discrete-time adaptation of the definition of dead time systems in \cite{halevi1996}. For input dead times \(\btau=[\tau_1~\cdots~\tau_m]\) and output dead times \(\btheta=[\theta_1~\cdots~\theta_p]\), let 
\begin{equation*}
\renewcommand{\arraystretch}{.9}
\bu_{\btau}(t)\!=\!\begin{bmatrix}u_1(t-\tau_1)&\cdots&u_m(t-\tau_m)\end{bmatrix}^\top\mkern-12mu,~
\by_{\btheta}(t)\!=\!\begin{bmatrix}y_1(t+\theta_1)&\cdots&y_p(t+\,\theta_p)\end{bmatrix}^\top
\end{equation*}
denote the delayed input and output signals, respectively. Analogously to \cref{eq:state}, the dynamics of a discrete-time dead time system are then characterized by
\begin{equation}
\label{eq:state-delay}
\begin{aligned}
    \bx(t+1)&=A_0\bx(t)+B_0\bu_{\btau}(t)\\
    \by_{\btheta}(t)&=C_0\bx(t)+D_0\bu_{\btau}(t),
\end{aligned}
\end{equation}
for \(A_0\in\R[n][n]\), \(B_0\in\R[n][m]\), \(C_0\in\R[p][n]\) and \(D_0\in\R[p][m]\). The system \(\sys_0=(A_0,B_0,C_0,D_0)\) models the dead time free dynamics of the system. In the following, we will always assume that \(\sys_0\) is minimal and free of dead time. Taking the \(\mathcal{Z}\)-transform of the above yields the transfer function
\begin{equation}
    \label{eq:tf-delay}
    \sys(z)=\Delta_{\btheta}(z)\sys_0(z)\Delta_{\btau}(z)=\Delta_{\btheta}(z)\bigl(C_0(zI_n-A_0)^{-1}B_0+D_0\bigr)\Delta_{\btau}(z),
\end{equation}
where \(\Delta_{\btau}(z)=\operatorname{diag}(z^{-\tau_1},\,\cdots,\,z^{-\tau_m})\) and \(\Delta_{\btheta}(z)=\operatorname{diag}(z^{-\theta_1},\,\cdots,\,z^{-\theta_p})\), reveals that dead times can be modelled multiplicatively in the frequency domain.

It should be noted that a dead time \(\delta\in\N\) manifests as a rational term \(z^{-\delta}\) in discrete-time, which means that the input and output dead time operators \(\Delta_{\btau}\) and \(\Delta_{\btheta}\) are rational functions of degree \(\bar{\tau}\coloneqq\sum_{i=1}^m\tau_i\) and \(\bar{\tau}\coloneqq\sum_{i=1}^p\theta_i\), respectively. In continuous-time, a dead time takes the form of an exponential term \(e^{\delta s}\) (with Laplace variable \(s\)) and, in consequence, \(\Delta_{\btau}\), \(\Delta_{\btheta}\) are both infinite dimensional. This is a key advantage of discrete-time models because the dead time operators can be directly realized as finite dimensional systems, whereas one has to approximate the exponential terms, e.g. by Pad\'{e} approximation \cite{partington2004}, in continuous-time.

The \ac{siso} pure dead time system with dead time \(0<\delta\in\N\) can be easily formulated in state-space by a canonical realization formula, e.g. controllable canonical form \cite{kailath1980} as
\begin{equation}
    \label{eq:ccf-delay}
    A_\delta=\!\begin{bmatrix}
        0 & 1 & 0 &  \phantom{\cdots}\\
        \vdots & \ddots & \ddots & 0\\
        \vdots &  & \ddots & 1\\
        0 & \cdots & \cdots & 0
    \end{bmatrix}\!\!\in\!\R[\delta][\delta],~
    B_\delta=\!\begin{bmatrix}
        0\\\vdots\\0\\1
    \end{bmatrix}\!\!\in\!\R[\delta][1],~
    C_\delta=\!\begin{bmatrix}
        1\\0\\\vdots\\0
    \end{bmatrix}^\top\mkern-12mu\!\!\in\!\R[1][\delta],~
    D_\delta=0.
\end{equation}

The \ac{mimo} input dead time operator \(\Delta_{\btau}=(A_{\btau},\,B_{\btau},\,C_{\btau},\,D_{\btau})\) can be written as a combined parallel system \cite{duke1986} of the \ac{siso} dead time system \cref{eq:ccf-delay} via
\begin{equation}
\label{eq:delta-realization}
\begin{aligned}
    A_{\btau}&=\blkdiag{A_{\tau_1},\,\cdots,\,A_{\tau_m}}\in\R[\bar{\tau}][\bar{\tau}],\\
    B_{\btau}&=\blkdiag{B_{\tau_1},\,\cdots,\,B_{\tau_m}}\in\R[\bar{\tau}][m],\\
    C_{\btau}&=\blkdiag{C_{\tau_1},\,\cdots,\,C_{\tau_m}}\in\R[m][\bar{\tau}],\\
    D_{\btau}&=\diag{D_{\tau_1},\,\cdots,\,D_{\tau_m}}\in\R[m][m],
\end{aligned}
\end{equation}
where \(D_{\tau_i}=1\) if \(\tau_i=0\). The output dead time operator \(\Delta_{\btheta}=(A_{\btheta},B_{\btheta},C_{\btheta},D_{\btheta})\) is realized analogously. It is easy to verify that
\begin{equation*}
    \begin{array}{c|c}
        A_{\delta}&B_{\delta}\\\hline C_{\delta}&D_{\delta}
    \end{array}^\hop
    \begin{array}{c|c}
        A_{\delta}&B_{\delta}\\\hline C_{\delta}&D_{\delta}
    \end{array}=I_{n+1},
\end{equation*}
i.e., all above-mentioned dead time systems are unitary. Specifically, it holds
\begin{equation*}
    \sys_{\delta}(e^{\imath\omega})^\hop\sys_{\delta}(e^{\imath\omega})=1
\end{equation*}
which reveals the allpass nature of a dead time transfer function. Furthermore, the Hankel singular values of a dead time are all equal to one. The reader is referred to \cite{gohberg1993} for a compilation of properties of unitary systems.
\section{Theory and calculation}
\label{sec:theory}
\subsection{Problem statement}
\label{sec:problem}
As mentioned earlier, the presence of dead times in the measurement data increases the order of the minimal realization and can drastically reduce the quality of reduced order approximations. Any system theoretic model reduction that exploits a decay of the spectrum of the Hankel operator must lead to large errors when applied to pure dead time systems, since their Hankel singular values are all equal to one. These errors manifest drastically in the time domain as so-called pre-ringing, a phase distortion also referred to as \textit{impulse smearing} \cite{adams2008,grantham2005}. In order to see exactly how an unstructured model \cref{eq:state} containing \ac{io} dead time is impacted, we expand the system concatenation \cref{eq:tf-delay}. According to \cite{duke1986}, and instating the notation of \cref{sec:dead-time}, the realization of the concatenated systems is given by
\begin{equation}
\label{eq:delay-realization}
    \begin{array}{c|c}
        A&B\\
        \hline
        C&D
    \end{array}
    \!\!=\!\!\begin{array}{c|c}
        \begin{matrix}
            A_{\btau}&0&0\\
            B_0C_{\btau}&A_0&0\\
            B_{\btheta}D_0C_{\btau}&B_{\btheta}C_0&A_{\btheta}
        \end{matrix}&
        \begin{matrix}
            B_{\btau}\\
            B_0D_{\btau}\\
            B_{\btheta}D_0D_{\btau}
        \end{matrix}\\
        \hline
        \begin{matrix}
            D_{\btheta}D_0C_{\btau}&D_{\btheta}C_0&C_{\btheta}
        \end{matrix}&
        D_{\btheta}D_0D_{\btau}
    \end{array}\!\!\in\!\R[(\bar{\tau}+n+\bar{\theta}+p)][(\bar{\tau}+n+\bar{\theta}+m)].
\end{equation}
Hence, the McMillan degree of the dead time system \((A,B,C,D)\) is \(n+\bar{\tau}+\bar{\theta}\), assuming that the interconnection does not result in a loss of controllability or observability \cite{schutter2000}. As all dead time related quantities \cref{eq:delta-realization} are sparse by definition, the realization in \cref{eq:delay-realization} possesses structure that enables efficient storage and computation. In a purely data-driven setting, however, it is challenging to obtain realizations of certain structures. In general, \ac{era} will produce reduced realizations that are (finitely) balanced \cite{akers1995} and therefore dense.
Yet, if the input and output dead times \(\btau\) and \(\btheta\) were known, the dead time operators \(\Delta_{\btau}\) and \(\Delta_{\btheta}\) could be readily constructed as in \cref{eq:delta-realization} and the dead-time-free system \(\sys_0\) could be computed with \ac{era} by removing the dead time from the measurement data \(h\), i.e. applying \ac{era} to dead-time-rectified measurement data \(h_0\in\ell_2^{p\times m}\) given by
\begin{equation}
\label{eq:truncation}
    h_0(t)=\begin{bmatrix}
        h_{11}(t+\tau_1+\theta_1)&\cdots&h_{1m}(t+\tau_m+\theta_1)\\
        \vdots&&\vdots\\
        h_{p1}(t+\tau_1+\theta_p)&\cdots&h_{pm}(t+\tau_m+\theta_p)
    \end{bmatrix}\in\R[p][m].
\end{equation}
The proposed approach not only preserves the dead time structure \cref{eq:delay-realization}, but also enables a dedicated reduction of solely the dynamical part of the system, thus minimizing pre-ringing errors. This naturally raises the question:
\begin{center}
\begin{tcolorbox}[enhanced,width=5in,center upper,
    fontupper=\normalsize,drop fuzzy shadow southeast,
    boxrule=.8pt,sharp corners,colframe=black
    %,colback=black!5
    ]
How can the dead times \(\btau\) and \(\btheta\) be determined from \(h\)?
\end{tcolorbox}
\end{center}
As we shall see, the answer is not straightforward and, in most cases, the dead times can only be approximated from the data. To better understand the challenge at hand, it is helpful to contrast the two different modelling approaches introduced earlier. A visualization of the different model structures is offered in \cref{fig:delay-model-comparison}. In the naive approach, the \ac{mimo} system is treated as an independent collection of \(p\cdot m\) \ac{siso} transmissions, each having their own dead time \(\delta_{ij}\), see \cref{fig:naive-delay}. As depicted in \cref{fig:ssm-delay}, the proposed \ac{mimo} state-space model of the form \cref{eq:state-delay} differs fundamentally in that all transmissions are modelled jointly and there are only \(m+p\) dead times present at the inputs and outputs.

Approximation techniques for the dead times \(\delta_{ij}\) of the parallel \ac{siso} model structure fall into the category of \ac{tde} to which a short overview will be given in the following.
\begin{figure}
    \centering
    \begin{subfigure}{.48\textwidth}
    \centering    
    \pgfmathsetmacro{\m}{3}
\pgfmathsetmacro{\p}{2}
\begin{tikzpicture}

\node at (-1.8,-1cm) {};
\node at (1.8+\bsize,\bsize+1.05cm) {};
\draw
\foreach \i in {1,...,\m} {
    node (in\i) at (-1.2cm,\pos{\i}{\m}) [draw,circle,label=left:\(u_{\i}\),inner sep=1.2pt] {}
};
\draw
\foreach \j in {1,...,\p} {
    node (out\j) at (1.2cm+\bsize,\pos{\j}{\p}) [draw,circle,label=right:\(y_{\j}\),inner sep=1.2pt] {}
};
\draw
\foreach \i in {1,...,\m} {
\foreach \j in {1,...,\p} {
node (d\i\j) at (\bsize/2,{.7cm*((\m-\i)*\p+(\p-\j+1)-(\m*\p/2)+1.3))}) [draw, circle, inner sep=1.2pt] {\footnotesize\(\delta_{\j\i}\)}
(in\i) -- (d\i\j.west) 
(d\i\j.east) -- (out\j)
}};
\end{tikzpicture}
    \caption{Model with independent \ac{siso} channels.}
    \label{fig:naive-delay}    
    \end{subfigure}%
    \begin{subfigure}{.48\textwidth}
    \centering
    \pgfmathsetmacro{\m}{3}
\pgfmathsetmacro{\p}{2}

\begin{tikzpicture}
\node at (-1.8,-1cm) {};
\node at (1.8+\bsize,\bsize+1.05cm) {};
\draw
node [myblock] (sys) at (\bsize/2,\bsize/2) {\Large\(\Sigma\)};
\draw
\foreach \i in {1,...,\m} {
    node (in\i) at (-1.2cm,\pos{\i}{\m}) [draw,circle,label=left:\(u_{\i}\),inner sep=1.2pt] {}
    node (din\i) at (-.5cm,\pos{\i}{\m}) [draw, circle, inner sep=2pt] {\footnotesize\(\tau_\i\)}
   (in\i) -- (din\i) -- ([yshift=(\pos{\i}{\m}-\bsize/2)]sys.west)
};
\draw
\foreach \j in {1,...,\p} {
    node (out\j) at (1.2cm+\bsize,\pos{\j}{\p}) [draw,circle,label=right:\(y_{\j}\),inner sep=1.2pt] {}
    node (dout\j) at (.5cm+\bsize,\pos{\j}{\p}) [draw, circle, inner sep=2pt] {\footnotesize\(\theta_\j\)}
    (out\j) -- (dout\j) -- ([yshift=(\pos{\j}{\p}-\bsize/2)]sys.east)
};
\end{tikzpicture} 
    \caption{\ac{mimo} state-space model.}    
    \label{fig:ssm-delay}    
    \end{subfigure}
    \caption{Schematic representation of different \ac{io} dead time \ac{lti} model structures for a \ac{mimo} system with four inputs \(u_1,u_2,u_3,u_4\) and three outputs \(y_1,y_2,y_3\).\label{fig:delay-model-comparison}}
\end{figure}
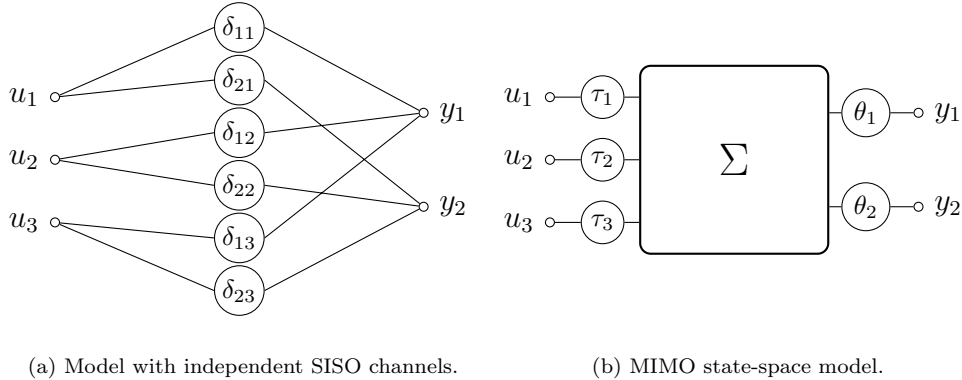
\subsection{Time delay estimation}
\label{sec:tde}
Given measurements of finite duration input and output trajectories \(\bu\in\ell_2^m\) and \(\by\in\ell_2^p\), the \ac{tde} problem consists of estimating the dead time (or delay) denoted by \(\delta_{ij}\) between all inputs and outputs. It is common to arrange all dead times in a so-called delay matrix
\begin{equation}
    \label{eq:delay-matrix}
    \bdelta=\begin{bmatrix}
        \delta_{11}&\cdots&\delta_{1m}\\
        \vdots&&\vdots\\
        \delta_{p1}&\cdots&\delta_{pm}
    \end{bmatrix}
\end{equation}
 which is the target of estimation. There exists a plethora of solution approaches, e.g., using generalized cross correlation in the time domain \cite{knapp1976} and in the time-frequency domain leveraging wavelets \cite{ni2010,tabaru2002,barsanti2003}; the Hilbert transform \cite{tamim2009,selvanathan2010}; or using higher order statistical moments \cite{usher2010}. Recently, more sophisticated approaches appeared based on copulae \cite{wang2020} and graph global smoothness \cite{wang2023}. The reader is referred to \cite{bjorklund2003} for a comprehensive survey of classical methods and \cite{defrance2008} for a comparative study of methods specifically considering \ac{rir} data. Since we consider \ac{ir} data, many of these approaches become simpler because we can set \(\by=h\) and choose an impulse input \(\bu\).

 \subsection{Dead time splitting}
 \label{sec:dts}
The problem setting considered in this work is more specific than the \ac{tde} problem because we are looking to determine \(\btau\) and \(\btheta\) rather than \(\bdelta\). Imposing the delay structure \cref{eq:state-delay} implies that \(\bdelta=\btheta\btau^\top\) which is equivalent to the following linear system of equations
\begin{equation}
 \label{eq:dtspre-problem}
    \begin{bmatrix}
    \boldone_{m}&\cdots&0&I_{m}\\
    \vdots&\ddots&\vdots&\vdots\\
    0&\cdots&\boldone_{m}&I_{m}
    \end{bmatrix}\begin{bmatrix}
        \btheta^\top\\\btau^\top
    \end{bmatrix}=
    F\bx
    =\vec{\bdelta},
\end{equation}
 where \(\boldone_n\in\R[n]\) denotes the vector of ones and \(\vec{\bdelta}=[
     \delta_{11}~\delta_{21}~\cdots~\delta_{(p-1)m}~\delta_{pm}]^\top\in\R[pm]\).
     
Similar formulations can be found in \cite{ni2010,haggblom2012}. However, they only consider the case where \(m=p=2\). This is a special case because it is the only \ac{mimo} configuration that always has a unique solution. In all other cases, \(m+p<mp\) and \cref{eq:dtspre-problem} will be overdetermined. 
Generally, it cannot be guaranteed that \cref{eq:dtspre-problem} is a consistent system of equations. Furthermore, in a data-driven setting, the structure of the latent dynamical system might not even be of the form \cref{eq:state-delay}. Therefore, it does not suffice to solve \cref{eq:dtspre-problem} by least-squares as proposed in \cite{ni2010} or by finding a rank-one approximation of \(\bdelta\).

In order to ensure that relevant dynamics are not pruned when performing the truncation step \cref{eq:truncation}, we impose the additional constraints \(\theta_i+\tau_j\leq\delta_{ij}\). This leads to the following linear program
\begin{equation}
    \label{eq:dts-problem}
\argmax_{\bx\geq0}\norm[1]{\bx}\qquad\text{subject to}\qquad F\bx\leq \vec{\bdelta},
\end{equation}
with \(F\in\R[pm][(p+m)]\) and \(\bx\in\R[p+m]\) as in \cref{eq:dtspre-problem}. We will refer to \cref{eq:dts-problem} as the \ac{dts} problem from here on out. A geometric depiction of the dead time splitting problem can be found in \cref{fig:dts} for an acoustical toy example. In \cref{fig:dts-hand}, the dead time splitting is done by hand arbitrarily. \cref{fig:dts-linprog} shows the solution of \cref{eq:dts-problem}. As can be seen, the dead times cannot be extracted entirely and some residual dead times remain. The sum of the residual dead times \(\norm[1]{F\bx-\vec{\bdelta}}\) is minimized by \cref{eq:dts-problem}.

\begin{figure}
    \centering
    \begin{subfigure}{0.48\textwidth}
    \scalebox{0.85}{\begin{tikzpicture}
\useasboundingbox (-3,-1.4) rectangle (4.35,2.4);

\begin{pgfonlayer}{foreground}
\coordinate[label=left:\(s_1\), draw, circle, fill=white, inner sep=1pt] (s1) at (-2.4,2);
\coordinate[label=left:\(s_2\), draw, circle, fill=white, inner sep=1pt] (s2) at (-0.8,0.6);
\coordinate[label=left:\(s_3\), draw, circle, fill=white, inner sep=1pt] (s3) at (-1.4,-1.2);
\coordinate[label=right:\(r_1\), draw, circle, fill=white, inner sep=1pt] (r1) at (3.8,2);
\coordinate[label=below:\(r_2\), draw, circle, fill=white, inner sep=1pt] (r2) at (2.5,.1);
\end{pgfonlayer}

%\draw let \p1=($(s1)-(r1)$), \n1={veclen(\x1,\y1)-149.53pt} in node at (0,0) {\n1};

% SOURCES
\foreach \i/\r in {1/149.53pt, 2/94.95pt, 3/116.97pt}{
    \path[name path=c] (s\i) circle (\r);
    \path[name path=l1] (s\i) -- (r1);
    \path[name path=lm] (s\i) -- ($(r1)!0.5!(r2)$);
    \path[name path=l2] (s\i) -- (r2);
    \path[name intersections = {of = l1 and c, by = sp1\i}];
    \path[name intersections = {of = lm and c, by = spm\i}];
    \path[name intersections = {of = l2 and c, by = sp2\i}];
    \draw[thick, dotted, index of colormap={1 of Set1}] (s\i) -- (sp1\i) to[arc through={clockwise,(spm\i)}] (r2) -- (s\i);
}

%RECEIVERS
\foreach \i/\r in {1/26.88pt}{
    \path[name path=c] (r\i) circle (\r);
    \path[name path=l1] (r\i) -- (s1);
    \path[name path=lm] (r\i) -- (s2);
    \path[name path=l2] (r\i) -- (s3);
    \path[name intersections = {of = l1 and c, by = rp1\i}];
    \path[name intersections = {of = lm and c, by = rpm\i}];
    \path[name intersections = {of = l2 and c, by = rp2\i}];
    \draw[thick, dotted, index of colormap={2 of Set1}] (r\i) -- (rp1\i) to[arc through={counterclockwise,(rpm\i)}] (rp2\i) (rpm\i) --(r\i);
}
\draw[thick,index of colormap={1 of Set1}] (s1) -- node[above] {\(\tau_1\)} (sp11);
\draw[thick,index of colormap={1 of Set1}] (s2) -- node[above] {\(\tau_2\)} (sp12);
\draw[thick,index of colormap={1 of Set1}] (s3) -- node[above, near start] {\(\tau_3\)} (sp13);
\draw[thick,index of colormap={2 of Set1}] (r1) -- node[below, near start] {\(\theta_1\)} (rp21);
\node[right, index of colormap={2 of Set1}] at (r2) {\(\theta_2=0\)};
\draw[thick,index of colormap={4 of Set1}] (sp12) -- (rpm1);
\draw[thick,index of colormap={0 of Set1}] (sp13) -- (rp21);

\end{tikzpicture}}
    \caption{Arbitrary splitting.\label{fig:dts-hand}}
    \end{subfigure}%
    \hfill
    \begin{subfigure}{0.48\textwidth}
    \scalebox{0.85}{\begin{tikzpicture}
\useasboundingbox (-3,-1.4) rectangle (4.35,2.4);

\begin{pgfonlayer}{foreground}
\coordinate[label=left:\(s_1\), draw, circle, fill=white, inner sep=1pt] (s1) at (-2.4,2);
\coordinate[label=left:\(s_2\), draw, circle, fill=white, inner sep=1pt] (s2) at (-0.8,0.6);
\coordinate[label=left:\(s_3\), draw, circle, fill=white, inner sep=1pt] (s3) at (-1.4,-1.2);
\coordinate[label=right:\(r_1\), draw, circle, fill=white, inner sep=1pt] (r1) at (3.8,2);
\coordinate[label=right:\(r_2\), draw, circle, fill=white, inner sep=1pt] (r2) at (2.5,.1);
\end{pgfonlayer}

%\draw let \p1=($(s1)-(r1)$), \n1={veclen(\x1,\y1)-149.53pt} in node at (0,0) {\n1};

% SOURCES
\foreach \i/\r in {1/1.39, 3/0.77}{
    \path[name path=c] (s\i) circle (\r);
    \path[name path=l1] (s\i) -- (r1);
    \path[name path=lm] (s\i) -- ($(r1)!0.5!(r2)$);
    \path[name path=l2] (s\i) -- (r2);
    \path[name intersections = {of = l1 and c, by = sp1\i}];
    \path[name intersections = {of = lm and c, by = spm\i}];
    \path[name intersections = {of = l2 and c, by = sp2\i}];
    \draw[thick, dotted, index of colormap={1 of Set1}] (s\i) -- (sp1\i) to[arc through={clockwise,(spm\i)}] (sp2\i) -- (s\i);
}

%RECEIVERS
\foreach \i/\r in {1/4.8, 2/3.34}{
    \path[name path=c] (r\i) circle (\r);
    \path[name path=l1] (r\i) -- (s1);
    \path[name path=lm] (r\i) -- ($(s1)!0.5!(s3)$);
    \path[name path=l2] (r\i) -- (s3);
    \path[name intersections = {of = l1 and c, by = rp1\i}];
    \path[name intersections = {of = lm and c, by = rpm\i}];
    \path[name intersections = {of = l2 and c, by = rp2\i}];
    \draw[thick, dotted, index of colormap={2 of Set1}] (r\i) -- (rp1\i) to[arc through={counterclockwise,(rpm\i)}] (rp2\i) --(r\i);
}

\draw[thick,index of colormap={1 of Set1}] (s1) -- node[below] {\(\tau_1\)} (sp21);
\node[below left,index of colormap={1 of Set1}, xshift=-.6cm] at (s2) {\(\tau_2=0\)};
\draw[thick,index of colormap={1 of Set1}] (s3) -- node[above] {\(\tau_3\)} (sp13);
\draw[thick,index of colormap={2 of Set1}] (r1) -- node[below, near start] {\(\theta_1\)} (rp21);
\draw[thick,index of colormap={2 of Set1}] (r2) -- node[above, near end] {\(\theta_2\)} (rp12);
\draw[dotted, thick, index of colormap={2 of Set1}] (r1) -- (s2) -- (r2);
\draw[thick,index of colormap={0 of Set1}] (sp21) -- (rp12);
\draw[thick,index of colormap={4 of Set1}] (sp13) -- (rp21);

\end{tikzpicture}}
    \caption{Solution to \cref{eq:dts-problem}.\label{fig:dts-linprog}} 
    \end{subfigure}
    \caption{The dead time splitting problem for a simple acoustic free field transmission in two dimensions with three sources \(s_1\), \(s_2\), \(s_3\) and two receivers \(r_1\), \(r_2\). The input dead times \(\tau_i\) are shown in blue and the output dead times \(\theta_i\) in green. The residual dead times are indicated by the orange and red lines, respectively.\label{fig:dts}} 
\end{figure}
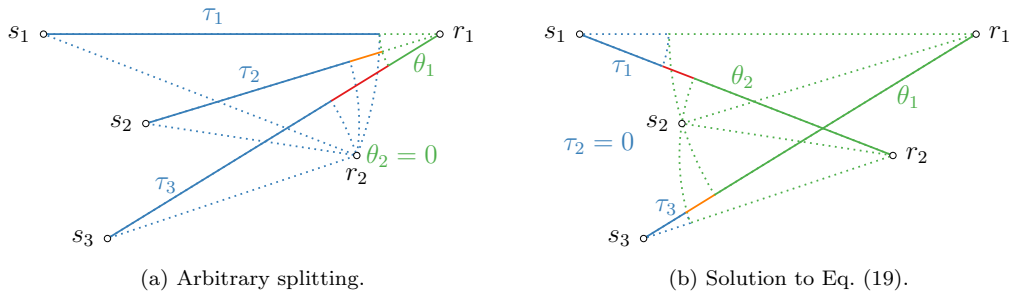

\subsection{Numerical implementation}
\label{sec:numerical-implementation}
In light of the considerations above, we derive a numerical procedure as follows:
Given \ac{ir} measurement data and the (estimated) dead times between each source and receiver,
\begin{enumerate}
    \item solve the \ac{dts} problem \cref{eq:dts-problem},
    \item remove the computed dead times by shifting the data accordingly,
    \item construct a \ac{rom} from the dead-time-rectified data with \ac{era},
    \item reapply the dead times to enforce a structured reduced order model \cref{eq:delay-realization}.
\end{enumerate}
Since \(F\) is sparse, \cref{eq:dts-problem} is straightforward to solve with an interior point method \cite{schork2020}. Therefore, the main numerical challenge is posed by the computation of the \ac{svd} of the Hankel matrix \cref{eq:tsvd} because it can become very large for real-world acoustical measurement data. As was demonstrated previously in \cite{pelling2021,kujawski2024} and according to the listed matrix sizes in \cref{tab:benchmarks}, a direct factorization becomes infeasible quickly and the Hankel matrix \(cH\) might not even fit into memory. To overcome these problems, the \ac{svd} can be substituted by an approximate matrix-free factorization. This concept was initially proposed in \cite{kramer2018} by employing CUR decompositions. Later, the use of \ac{rsvd} for \ac{era} has been investigated in \cite{minster2021}. Recently, tensor-based decompositions such as the Tucker decomposition were considered in \cite[Sec. 6.1.]{kilmer2022}.

In this work, we employ randomized \ac{era} from \cite{minster2021} since it performed remarkably well for the considered benchmarks in the past \cite{pelling2021,kujawski2024}. In the standard formulation of \ac{era}-type algorithms, the desired model order has to be chosen a priori. From an application standpoint, this is not ideal because the computations need to be repeated from scratch whenever the reduced order model does not possess the desired accuracy. As demonstrated in the following section, even with randomization techniques, the computational effort can be substantial for large datasets. Therefore, we reformulate randomized \ac{era} as an adaptive pipeline which can improve the approximation quality by drawing additional samples of the Hankel operator and reuses calculations of previous iterations with a QR-updating scheme.
Additionally, we equip randomized \ac{era} with a \ac{loo} error estimator \cite{epperly2024} with which the approximation quality of the \ac{rsvd} can be assessed cheaply.

For sake of brevity, the proposed adaptive randomized \ac{era} reformulation is summarized as a high-level routine in \cref{algo:randera}. The details of the numerical implementation and all subroutines, in particular the proposed communication-avoiding QR routine \cite{fukaya2020} and \ac{loo} error estimator \cite{epperly2024}, can be found in \cref{sec:algos}.
\begin{algorithm}
\caption{Adaptive randomized \ac{era} based on \cite{minster2021}. \label{algo:randera}}
    \begin{algorithmic}[1]
        %\label[randera]{randera}
        \Require{Sequence of Markov parameters \(h=\{h_k\}_{k=0}^{2s-1}\) with \(h_k\in\R[p][m]\) that fulfill \cref{eq:assumption-kung}, approximation tolerance \(\gamma>0\), blocksize \(b\in\N\), number of power iterations \(q\in\N_0\).}
        \Ensure{Realization \(\sys_r=(A_r,B_r,C_r,D_r)\) satisfying \cref{eq:err-est}}
        \State \(\cH\gets \{h_k\}_{k=1}^{2s}\) \Comment{Implicit representation by \cite[Algorithm 3.1]{minster2021}.}
        \State \(\etol\gets\gamma\norm[\htwo,\eta]{G}\) \Comment{Scale the tolerance (see \cref{sec:err-est})
        .}
        \State \(U,\bSigma,V\gets\textsc{RandSVD}(\cH,b,q,\etol)\) \Comment{\cref{algo:randsvd}}
        \State \(U_f\gets U(p+1:ps,\,:)\)
        \State \(U_l\gets U(1:p(s-1),\,:)\)
        \State \(A\gets \bSigma^{-1/2}U_f^\dagger U_l\bSigma^{1/2}\)
        \State \(B\gets \bSigma^{1/2}V(1:m,\,:)^\top\)
        \State \(C\gets U(1:p,\,:)\bSigma^{1/2}\)
        \State \(D\gets h_0\)
    \end{algorithmic}
\end{algorithm}
\subsection{Proposed error estimator}
\label{sec:err-est}
The adaptive formulation of randomized \ac{era} in \cref{algo:randera} identifies a \ac{rom} based on the  \ac{rsvd} tolerance parameter \(\etol\) of \cref{algo:randsvd}. Since the \ac{loo} error estimator \(\eloo\) from \cite{epperly2024} (see \cref{algo:looest}) estimates the approximation error of the \ac{rsvd} of the Hankel matrix in the sense that
\begin{equation}
\label{eq:eloo-expec}
    \mathbb{E}\bigl(\norm[\mathrm{F}]{\smash{\cH-U_r\bSigma_rV_r^\top}}\bigr)=\eloo,
\end{equation}
where \(\mathbb{E}(\,\cdot\,)\) denotes the expected value, and according to \cref{algo:randsvd} we have \(\eloo<\etol\), it remains to relate \(\eloo\) to the \(\cH_2\)-error of the \ac{rom}. We can achieve this by following considerations made in \cite{kilmer2022}. To this end, let
\begin{equation*}
    \eta_k=\begin{cases}
        1,&k=0,\\
        k,&1\leq k\leq s,\\
        2s-k,&s<k<2s,
    \end{cases}
\end{equation*}
denote the absolute frequency of elements of the (block-)Hankel matrix. With \cref{eq:parseval}, let
\begin{equation}
\label{eq:weighted-h2-norm}
\norm[\htwo,\eta]{\sys}=\biggl(\,\sum_{k=0}^\infty\eta_k\norm[\mathrm{F}]{h_k}^2\biggr)^{1/2}
\end{equation}
denote the time-weighted \(\htwo\)-norm with weights \(\eta_k\). Since \(1<\eta_k\leq s\) for all \(k\), it holds
\begin{equation*}
    \frac{1}{\sqrt{s}}\norm[\htwo,\eta]{\sys}=\!\biggl(\,\sum_{k=0}^\infty\frac{\eta_k}{s}\norm[\mathrm{F}]{h_k}^2\biggr)^{1/2}\!\!\leq\norm[\cH_2]{\sys}\leq\!\biggl(\,\sum_{k=0}^\infty\eta_k\norm[\mathrm{F}]{h_k}^2\biggr)^{1/2}\!\!=\norm[\htwo,\eta]{\sys},
\end{equation*}
which means that both norms are equivalent. By \cref{eq:hankel-matrix,eq:weighted-h2-norm} it holds
\begin{equation*}
\norm[\htwo,\eta]{\sys}=\bigl(\norm[\mathrm{F}]{\cH}^2+\norm[\mathrm{F}]{h_0}^2\bigr)^{1/2}.
\end{equation*}
Since, by definition, \(h_0=D=D_r\), the approximation error can be bounded by
\begin{equation}
\label{eq:triangle-ineq}
\begin{aligned}
\norm[\htwo]{G-G_r}\leq\norm[\htwo,\eta]{\sys-\sys_r}&=\norm[\mathrm{F}]{\cH-\cH_r}\\&\leq\norm[\mathrm{F}]{\smash{\cH-U_r\bSigma_rV_r^\top}}+\norm[\mathrm{F}]{\smash{\cH_r-U_r\bSigma_rV_r^\top}},
\end{aligned}
\end{equation}
using the triangle inequality for the last part.

Ideally, we would like to obtain an efficient error estimator that is fast to evaluate. To achieve this, we need to simplify \cref{eq:triangle-ineq} significantly. Firstly, we drop \(\norm[\mathrm{F}]{\smash{\cH_r-U_r\bSigma_rV_r^\top}}\), since it requires a construction of \(\sys_r\), an \ac{ir} computation, and the construction of \(\cH_r\). Each individual operation alone is numerically infeasible during error estimation\footnote{Note that this simplification is reasonable, since the approximation \cref{eq:triangle-ineq} is conservative and, as will be demonstrated in \cref{sec:error-bounds-and-estimators}, the approximation error follows the \ac{rsvd}-error \(\norm[\mathrm{F}]{\smash{\cH-U_r\bSigma_rV_r^\top}}\) asymptotically in practice.}. Leveraging \cref{eq:eloo-expec,eq:triangle-ineq}, yields the following approximation of the expectation of the relative error:
\begin{equation}
\label{eq:eloo-approx}
\mathbb{E}\biggl(\,\frac{\norm[\htwo]{\sys-\sys_r}}{\norm[\htwo]{\sys}}\biggr)=\frac{\mathbb{E}(\norm[\htwo]{\sys-\sys_r})}{\norm[\htwo]{\sys}}\leq\frac{\eloo}{\norm[\htwo]{\sys}}.
\end{equation}
Observe that for \(r\rightarrow0\), \(\eloo\rightarrow\norm[\mathrm{F}]{\cH}=\norm[\htwo,\eta]{\sys}\). By assuming that the variance of \(\erel\) is negligible and applying a normalization to the right-hand side of \cref{eq:eloo-approx}, we obtain a rough estimator for the relative error:
\begin{equation}
    \label{eq:err-est}
    \frac{\norm[\htwo]{\sys-\sys_r}}{\norm[\htwo]{\sys}}\approx\frac{\eloo}{\norm[\htwo,\eta]{\sys}}.
\end{equation}
As desired, it holds
\begin{equation*}
\lim_{r\rightarrow0}\frac{\eloo}{\norm[\htwo,\eta]{\sys}}=1
\end{equation*}
and the estimator does not rely on any computations involving \(\sys_r\). Further, \(\norm[\htwo,\eta]{\sys}\) needs to be computed only once directly from the data. In conclusion, any desired tolerance \(\gamma>0\) for the relative error of the reduced order model can be prescribed by simply setting \(\etol=\gamma\norm[\htwo,\eta]{\sys}\) when computing the \ac{rsvd}.
\section{Benchmarks}
\label{sec:benchmarks}
To validate the proposed dead time extraction scheme as well as the adaptive formulation of randomized \ac{era}, our method is applied to several benchmark datasets that are briefly introduced below. Numerical results are reported and discussed in \cref{sec:results}.

All numerical simulations were performed on a machine with two 12-core Intel\regtm{} Xeon\regtm{} Silver 4214R CPUs with 256 GB RAM running on Ubuntu Linux 22.04.5 LTS. \cref{algo:randera} was implemented in the Python programming language (version 3.12.7) using the open-source \ac{mor} library pyMOR \cite{milk2016,balicki2019} (version 2024.2) for its implementations of adaptive \ac{rsvd} and the \ac{loo} error estimator on top of which randomized \ac{era} was implemented on a development branch. We intend to include the proposed adaptive randomized \ac{era} algorithm in the upcoming pyMOR release (version 2025.1). Additionally, a custom routine for efficient matrix-vector products with block-Hankel matrices was implemented that is accelerated by just-in-time compilation with the Python compiler Numba \cite{lam2015} (version 0.61.2). Due to the size of the benchmarks, it was necessary to use an implementation of \ac{blas} with \(64\)-bit integer (ILP64) support to enable the computation of the larger \acp{rom}. Further, all computations are done in single precision (\(32\)-bit floating point) to reduce memory usage.
\begin{center}%
  \setlength{\fboxsep}{5pt}%
  \fbox{%
  \begin{minipage}{.92\linewidth}
    \textbf{Source code availability}\newline
    The source code and scripts used to compute the results presented in this
    paper can be obtained from
    \begin{center}
      \href{https://doi.org/10.5281/zenodo.15586170}%
        {\texttt{10.5281/zenodo.15586170}}
    \end{center}
    under the MIT licence and authored by Art J. R. Pelling.
  \end{minipage}}
\end{center}
We now briefly introduce the \ac{ir} datasets that will serve as benchmark problems for our method.
\subsection{Microphone Array Impulse Response Database for Acoustic Learning}
The \ac{miracle} \cite{kujawski2024} consists of four acoustic scenarios of \acp{rir} that we measured in the anechoic chamber at TU Berlin. Each scenario comprises spatially distributed \acp{ir} that were measured with a planar microphone array of \(64\) microphones. The sources were arranged on square equidistant grids parallel to the microphone array plane, with a grid size of \(23\mm\) and \(5\mm\) depending on the scenario. The dataset includes scenarios at source-plane-receiver-plane distances of \(73.4\m\) and \(146.7\m\), respectively. The arrangement of the measurement setup is depicted in \cref{fig:miracle}. 

\ac{miracle} is specifically designed to serve as a large-scale benchmark problem for reduced-order modelling algorithms, among others. A short application example with \ac{era} is given in \cite[Sec. 4.2.2]{kujawski2024} as a proof of concept. Here, we will analyse the performance of \ac{era} more thoroughly and improve upon the results reported previously. 

Running \ac{era} for the larger scenarios in the \ac{miracle} dataset; namely A1, A2, and R2; exhaust the memory of our machine for \acp{rom} with an order of \(2000\) and above. To avoid out-of-core computations, we reduce the size of the dataset by sub-selecting a coarser grid of source positions. The size of the reduced scenarios is only a quarter of the original scenarios and are obtained by omitting every other source position in both the \(x\)- and \(y\)-direction. These coarser scenarios are denoted by a “-C1” suffix, i.e. A1-C1, A2-C1, and R2-C1. Even with this reduction, the coarser benchmark cases can be considered rather large, as \cref{tab:benchmarks} reveals.
\begin{figure}
    \centering
    \includegraphics[width=.8\textwidth]{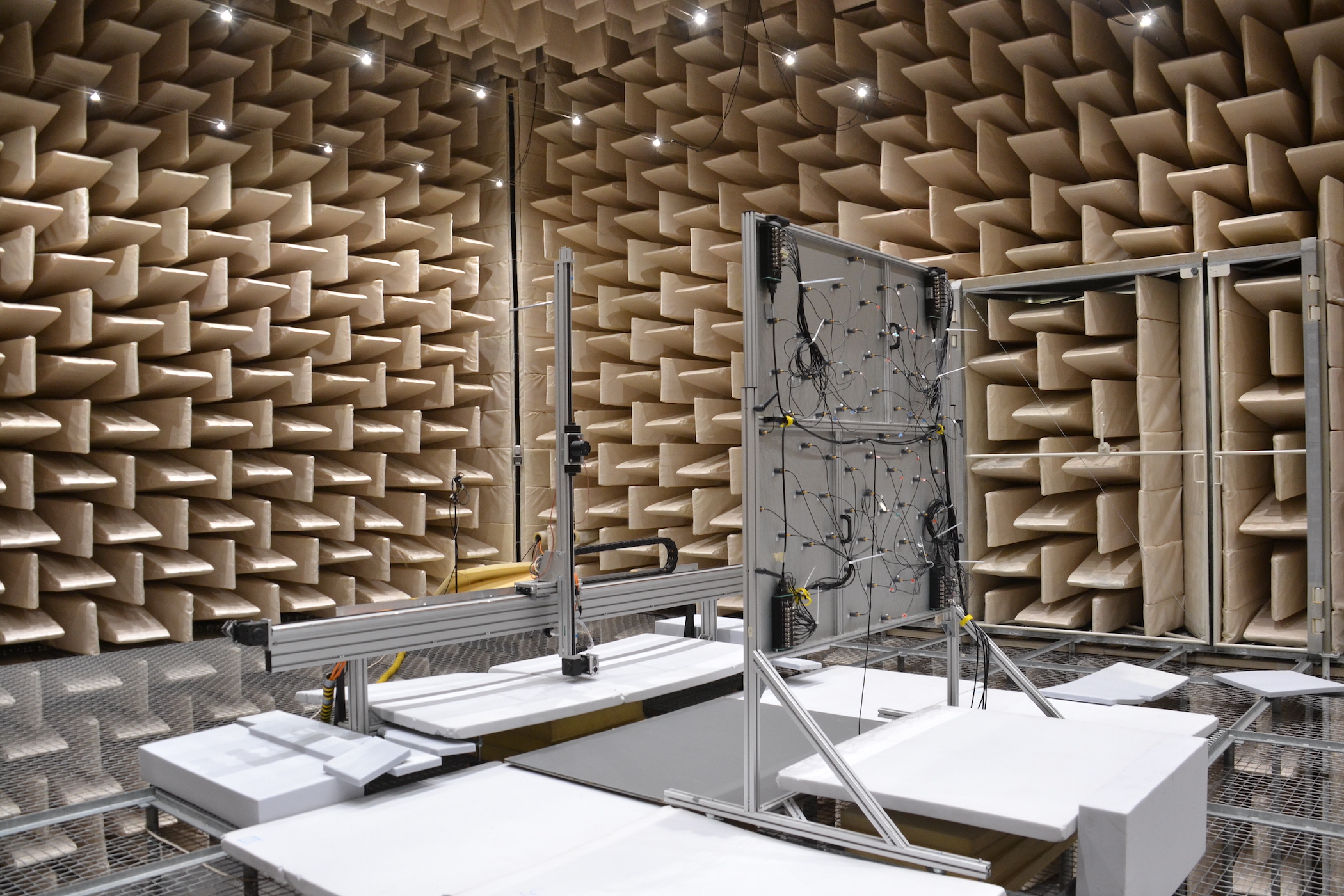}
    \caption{Geometrical setup of the \ac{miracle} scenarios. The sources (positionable loudspeaker) are distributed on a planar grid parallel to a planar microphone array. The image with annotations removed is taken from \cite{kujawski2024} under the \href{https://creativecommons.org/licenses/by/4.0/}{CC-BY 4.0 licence}. No further changes were made.}
    \label{fig:miracle}
\end{figure}
\subsection{Multi-Channel Impulse Response Database}
Furthermore, we revisit the \ac{mird} \cite{hadad2014}, which was used as a demonstrator for reduced order modelling in an earlier work by the authors \cite{pelling2021}, and recently also considered in \cite{hilgemann2023}. The \ac{mird} contains \ac{rir} measurements of a square room (\(6\m\times6\m\times2.4\m\)) with configurable absorber panels to control the reverberation time \RT. The \acp{rir} were measured with a linear microphone array consisting of \(8\) microphones for \(26\) source positions that lie on two semicircles around the microphone array. The source positions are distributed equiangularly in \(15^{\circ}\) steps, with the radii of the semicircles being \(1\m\) and \(2\m\), respectively. The database contains several acoustical scenarios with three different reverberation times \(\RT\in\{0.16\seconds,\,0.36\seconds,\,0.61\seconds\}\) and varying microphone spacings of the microphone array. We consider the scenarios with inter-microphone spacings of \(3\cm\) which we denote as SHORT3, MID3, and LONG3, in reference to the different reverberation times \RT. With a duration of \(10\seconds\), i.e. \(480000\) samples, the supplied \acp{rir} in are overly long compared to the reverberation times of the scenarios. Furthermore, the \acp{rir} contain additional dead time of \(629\) samples that is not explained by source-receiver distances but rather by a processing delay in the measurement setup. To avoid unnecessary computations, the \acp{rir} are truncated by removing the unphysical dead time at the start and adjusting the overall length  to the respective reverberation times before applying our method. The number of samples in the \acp{rir} of each considered scenario can be found in \cref{tab:benchmarks}.
\begin{figure}
    \centering
    \includegraphics[width=.8\textwidth]{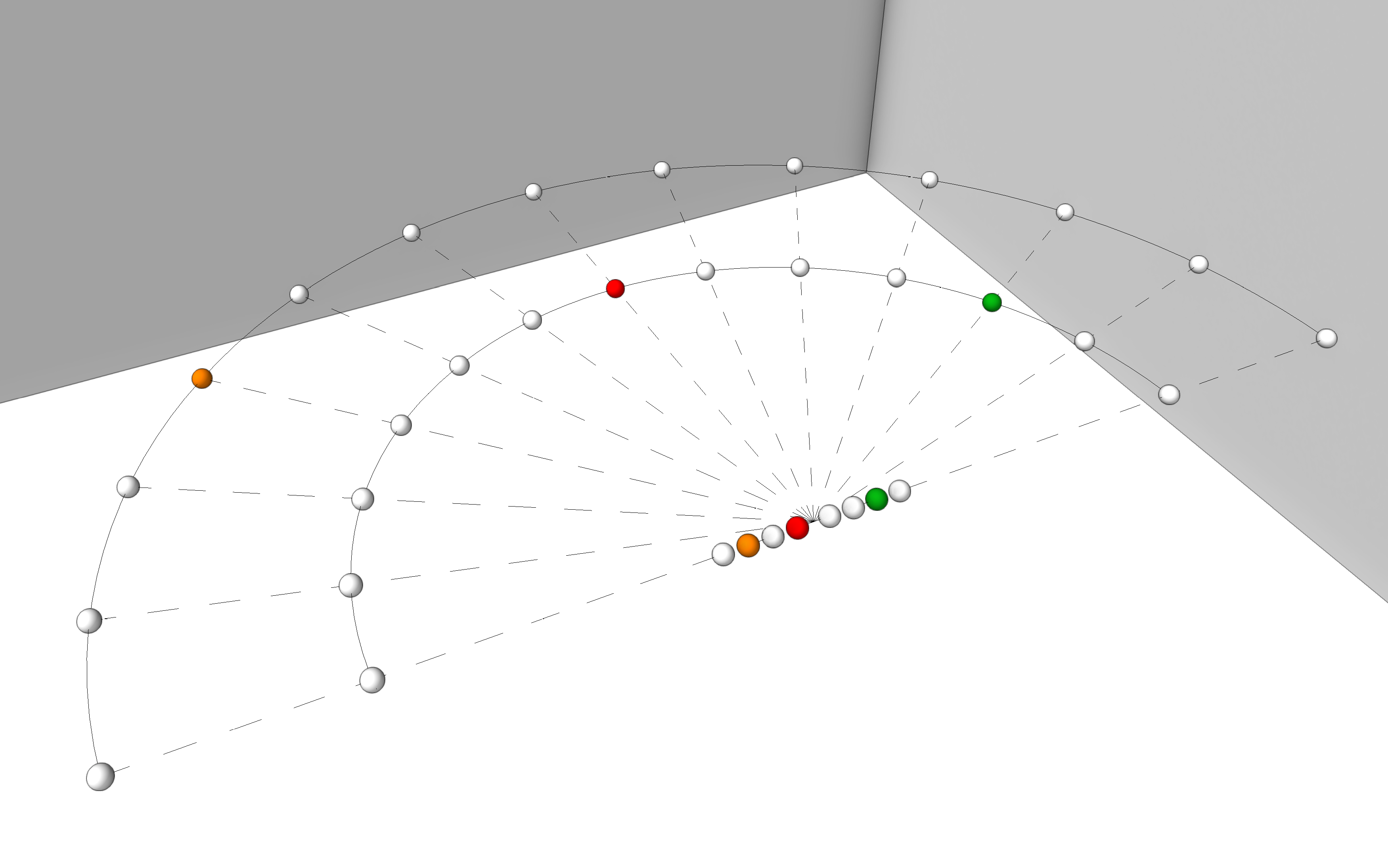}
    \caption{Geometrical setup of the \ac{mird} scenarios. The \(26\) sources are equiangularly distributed on two concentric semicircles with a linear array of eight microphones at the centre. The image is taken from \cite{pelling2021} under the \href{https://creativecommons.org/licenses/by/4.0/}{CC-BY 4.0 licence}. No changes were made.}
    \label{fig:mird}
\end{figure}
\begin{table}
\centering
\begin{tabular}{l|l|c|c|c|c|c}
     Dataset & Scenario &\(m\) & \(p\) & \(s\) & \(f_s\) & \(\operatorname{dim}(\cH)\)\\
     \hline
     \acs{miracle} & D1 & \(1089\) & \(64\) & \(1024\) & \(32\kHz\) & \(65536\times1115136\)\\     
     \acs{miracle} & A1-C1 & \(1024\) & \(64\) & \(1024\) & \(32\kHz\) & \(65536\times1048576\)\\     
     \acs{miracle} & A2-C1 & \(1024\) & \(64\) & \(1024\) & \(32\kHz\) & \(65536\times1048576\)\\   
     \acs{miracle} & R2-C1 & \(1024\) & \(64\) & \(1024\) & \(32\kHz\) & \(65536\times1048576\)\\      
%     \acs{miracle} & A1 & \(4096\) & \(64\) & \(1024\) & \(32\kHz\) & \(65536\times4194304\)\\     
%     \acs{miracle} & A2 & \(4096\) & \(64\) & \(1024\) & \(32\kHz\) & \(65536\times4194304\)\\   
%     \acs{miracle} & R2 & \(4096\) & \(64\) & \(1024\) & \(32\kHz\) & \(65536\times4194304\)\\        
     \acs{mird} & SHORT3 & \(26\) & \(8\) & \(7680\) & \(48\kHz\) & \(61440\times199680\)\\
     \acs{mird} & MID3 & \(26\) & \(8\) & \(17280\) & \(48\kHz\) & \(138240\times449280\)\\
     \acs{mird} & LONG3 & \(26\) & \(8\) & \(29280\) & \(48\kHz\) & \(234240\times761280\)\\
\end{tabular}
\caption{Properties for the considered benchmark scenarios. \(m\): number of sources, \(p\): number of receivers, \(s\): number of samples per \ac{ir}, \(f_s\): sampling frequency, \(\operatorname{dim}(\cH)\): dimension of the Hankel matrix.\label{tab:benchmarks}}
\end{table}

\section{Results and discussion}
\label{sec:results}

For each of the benchmark scenarios listed in \cref{tab:benchmarks}, several \acp{rom} of increasing order and accuracy were created according to procedure described in \cref{sec:numerical-implementation}. The accuracy of the \acp{rom} will be assessed with the following relative error metric in dB
\begin{equation}
\label{eq:erel}
    \erel=20\operatorname{log}_{10}\biggl(\,\frac{\norm[\cH_2]{\sys-\sys_r}}{\norm[\cH_2]{\sys}}\biggr)=10\operatorname{log}_{10}\biggl(\,\frac{\sum_{k=1}^{2s-1}\norm[\mathrm{F}]{\smash{h_k-C_rA_r^{k-1}B_r}}^2}{\sum_{k=1}^{2s-1}\norm[\mathrm{F}]{h_k}^2}\biggr).
\end{equation}
In this section, the parameter \(r\) in \cref{eq:erel} refers to the model order of the unstructured parts, i.e. the dimension of \(A_0\), as this is a more accurate indicator of the associated cost of the model regarding storage and computation. The model order of the structured system obtained by the interconnection of input and output dead time systems \cref{eq:delay-realization} will usually be much larger. However, the dead times can be realized efficiently by suitable delay lines at a computational cost that is negligible in comparison to that of the reduced order model. Thus, whenever we speak of the model order \(r\) in the following, we refer to the dimension of \(\range{A_0}\).

\subsection{Error bounds and estimators}
\label{sec:error-bounds-and-estimators}
The performance of our proposed error estimator \cref{eq:err-est}, denoted by
\begin{equation}
    \label{eq:eest}
    \eest=20\operatorname{log}_{10}\biggl(\,\frac{\eloo}{\norm[\htwo,\eta]{G}}\biggr),
\end{equation}
is visualized in \cref{fig:bounds} alongside Kung's bound. The definition of Kung's bound in \cref{eq:kung-bound} applies to the absolute error. Thus, \cref{fig:bounds} depicts the normalized corrected bound given by
\begin{equation}
\label{eq:ekc}
\ekc=20\operatorname{log}_{10}\biggl(\,\frac{\sqrt{r+m+p}\cdot\sigma_{r+1}(\cH)}{\norm[\htwo]{\sys}}\biggr)
\end{equation}
and the normalized erroneous bound given by
\begin{equation}
\label{eq:ekw}
\ekw=10\operatorname{log}_{10}\biggl(\,\frac{\sqrt{r+m+p}\cdot\sigma_{r+1}(\cH)}{\norm[\htwo]{\sys}^2}\biggr).
\end{equation}
As reported in \cite[Section 4]{kramer2016} and \cite{kramer2018}, the bound of Kung \(\ekw\) is not sharp and can overestimate the actual error substantially. For larger model orders, this overestimation increases with model order because the decay of \(\erel\) and \(\ekw\) do not match asymptotically. Further, the bound is violated for the D1 scenario for model orders below \(1000\) (see \cref{fig:bounds-d1}). Even though this violation of Kung's bound cannot formally be considered a counterexample of \cref{thm:era} because we leverage randomization and, therefore, all bounds only hold in expectation, we are convinced that the violation does not stem from an approximation error due to randomization but see it as further evidence that the bound of Kung as stated in \cite{kung1978,kramer2016,kramer2018} is indeed flawed.

The corrected version of Kung's bound does not violate \cref{thm:era} and its difference in decay compared to \(\erel\) is reduced, albeit also increasing slightly with the model order. Unfortunately, \(\ekc\) is much larger overall, substantially amplifying the bound's overestimation of \(\erel\). In consequence, both bounds \(\ekw\) and \(\ekc\) are not particularly suited for practical applications.

Although it is not an upper bound in the formal sense, our proposed error estimator \(\eest\) mostly overestimates the actual error \(\erel\) by about \(2\)-\(6\)dB for the benchmarks we consider here (see also \cref{fig:model-comparison}). Furthermore, the decay of the estimator roughly matches the decay of the actual error, in particular, the overestimation does not worsen with model order opposed to the bound of Kung.
\begin{figure}
\centering
\begin{subfigure}{.5\textwidth}
    %\captionsetup{aboveskip=-1em}
    \centering
    \includegraphics{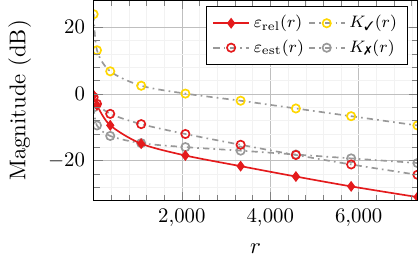}
    \caption{D1 scenario.\label{fig:bounds-d1}}
\end{subfigure}%
\begin{subfigure}{.5\textwidth}
    %\captionsetup{aboveskip=-1em}
    \hfill
    \includegraphics{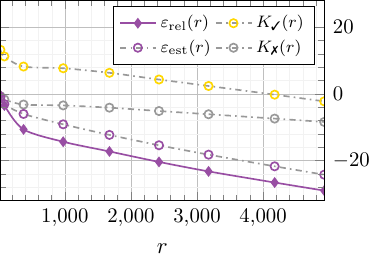}
    \caption{SHORT3 scenario.}
\end{subfigure}
\caption{Comparison of relative error, bounds, and estimators. The figure depicts the relative error \(\erel\) in \cref{eq:erel}, the erroneous and corrected version of Kung's bound \(\ekw\) in \cref{eq:ekw} and \(\ekc\) in \cref{eq:ekc}, respectively, and the proposed error estimator \(\eest\) in \cref{eq:eest} over model order for the D1 scenario and SHORT3 scenario in dB.}
\label{fig:bounds}
\end{figure}
\subsection{Dead time extraction}
\label{sec:dte}
We now focus our attention on the performance of the proposed dead time extraction via the solution of the associated \ac{dts} problem \cref{eq:dts-problem}. So far, the problem of dead time extraction is not well-studied. The only other approach for dead time extraction the authors are of is the removal of the least common dead time which was pursued by the authors previously \cite{pelling2021,kujawski2024} as well as in \cite{hilgemann2023,jalmby2021}. In this obvious approach, the least common dead time present in all \ac{siso} transmission paths given by
\begin{equation*}
    \dlc=\!\min_{\substack{i\in\{1,\,\dots,\,p\}\\j\in\{1,\,\dots,\,m\}}}\{\delta_{ij}\}
\end{equation*}
is extracted. In terms of the notation introduced in \cref{sec:dt-ssm}, this entails splitting the unstructured system into a dead-time-rectified system \(\sys_0\) and a single pure dead time system \(\Delta_{\scriptscriptstyle\mathrm{LC}}=\diag{\dlc,\,\dots,\,\dlc}\in\R[\min\{m,p\}]\) that is chained either to the input or output of \(\sys_0\) depending on its dimension to ensure a minimal representation.

Therefore, we compare the proposed dead time extraction method to the extraction of the least common dead time, as well as omitting the extraction altogether. To facilitate a meaningful comparison, the relative error \cref{eq:erel} is depicted over the \acp{dof} of the model rather than model order in \cref{fig:dtc}. For a \ac{rom} without dead time extraction, the \acp{dof} are given by 
\begin{equation}
\label{eq:dofsnone}
    \dofsnone=(r+p)\cdot(r+m).
\end{equation} For our method, the \acp{dof} are given by 
\begin{equation}
\label{dofs:dts}
\dofsdts=\dofsnone+\sum_{i=1}^p\theta_i+\sum_{j=1}^m\tau_j,
\end{equation}
i.e. the number of non-zero entries in \cref{eq:delay-realization}. For the extraction of the least common dead time, we have
\begin{equation}
\label{eq:dofslc}
    \dofslc=\dofsnone+\min\{m,p\}\cdot\dlc.
\end{equation}

Our numerical experiments reveal that an extraction of the least common dead time achieves the same relative error with less \acp{dof} compared to an omission of dead time extraction. Further, our proposed \ac{dts}-based dead time extraction method strictly outperforms both other considered methods.

\begin{figure}[p]
\centering
\begin{subfigure}{.5\textwidth}
    %\captionsetup{aboveskip=-1em}
    \centering
    \includegraphics{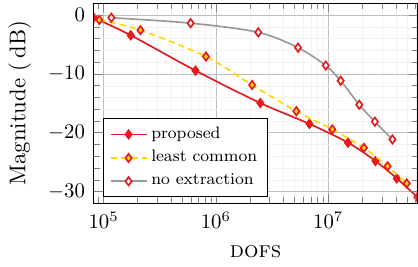}
    \caption{D1 scenario.\label{fig:dtc-miracle-d1}}
\end{subfigure}\vspace*{1em}%
\begin{subfigure}{.5\textwidth}
    %\captionsetup{aboveskip=-1em}
    \hfill
    \includegraphics{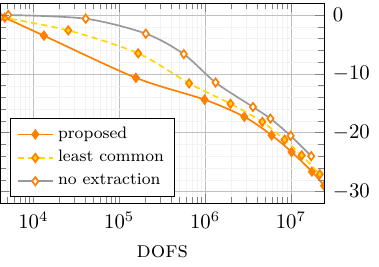}    
    \caption{SHORT3 scenario.\label{fig:dtc-mird-short3}}
\end{subfigure}
\begin{subfigure}{.5\textwidth}
    %\captionsetup{aboveskip=-1em}
    \centering
    \includegraphics{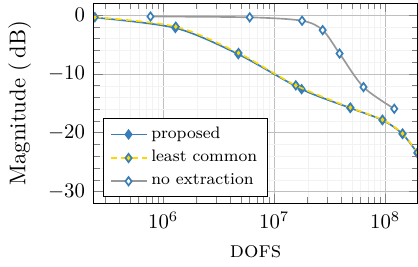}    
    \caption{A1-C1 scenario.\label{fig:dtc-miracle-a1red}}
\end{subfigure}\vspace*{1em}%
\begin{subfigure}{.5\textwidth}
    %\captionsetup{aboveskip=-1em}
    \hfill
    \includegraphics{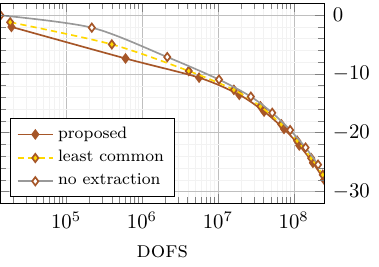}    
    \caption{MID3 scenario.\label{fig:dtc-mird-mid3}}
\end{subfigure}
\begin{subfigure}{.5\textwidth}
    %\captionsetup{aboveskip=-1em}
    \centering
    \includegraphics{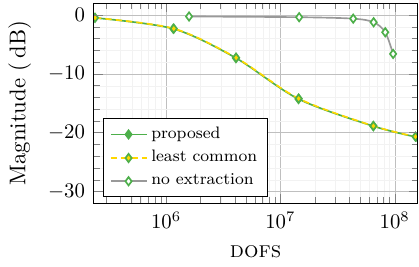}    
    \caption{A2-C1 scenario.\label{fig:dtc-miracle-a2red}}
\end{subfigure}%
\begin{subfigure}{.5\textwidth}
    %\captionsetup{aboveskip=-1em}
    \hfill
    \includegraphics{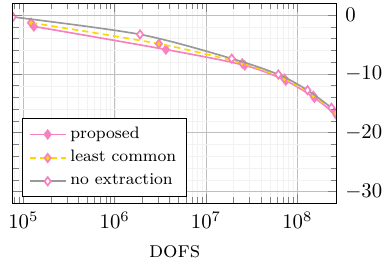}    
    \caption{LONG3 scenario.\label{fig:dtc-mird-long3}}
\end{subfigure}
\caption{Comparison of dead time extraction methods for selected benchmark scenarios. The plot depicts the relative error in dB over \acp{dof} for the proposed method based on the \ac{dts} problem, the extraction of the least common dead time, and no extraction. Diamonds indicate orders for which a model was constructed.\label{fig:dtc}}
\end{figure}

Interestingly, the reduction of \acp{dof} of our method over the removal of the least common dead time is less pronounced for D1 than for SHORT3, as can be seen in \cref{fig:dtc}. This phenomenon can be explained by the geometrical arrangement of sources and receivers in the scenarios: For \ac{miracle}, the sources and receivers are arranged on two parallel planes, whereas with \ac{mird}, the sources are placed on two semicircles around the linear array of receivers.
Roughly speaking, an extraction of the least common dead time will compensate for the dead time caused by the radius of the inner semicircle with \ac{mird}. A \ac{dts}-based extraction can effectively also extract the dead time caused by the radius of the outer semicircle. \cref{fig:dtc-mird-short3,fig:dtc-mird-mid3,fig:dtc-mird-long3} reveal that the reduction of required \acp{dof} is about twice as large with proposed \ac{dts}-based extraction over the extraction of the least common dead time for lower model orders of the \ac{mird} benchmarks.
On the contrary, this property is not observed for the D1 scenario in \cref{fig:dtc-miracle-d1} because the \ac{dts}-based method yields diminishing returns, once the least common dead time is extracted for a planar arrangement of sources and receivers. As the extent of the source plane increases, extracted dead times of the \ac{dts}-based method approach the least common dead time. As a result, both methods lead to very similar results for planar arrangements where the extent of both planes is equal. This effect is clearly recognisable in \cref{fig:dtc-miracle-a1red,fig:dtc-miracle-a2red}.
\subsection{Computational performance}
The overall performance of the proposed adaptive randomized \ac{era} method in \cref{algo:randera} is very positive for the considered benchmarks both regarding \ac{rom} accuracy and computational efficiency.
The asymptotic decay of the relative error and the proposed error estimator is mainly dictated by the complexity of dynamics, i.e. the number of (dynamically independent) inputs and outputs, the decay of the \acp{ir}, and the amount of dead time in the data. As outlined in the previous subsection, our proposed dead time extraction scheme can extract most dead times but depending on the scenario, residual dead times can remain.
Considering \cref{fig:overview-miracle}, one can see that the error decay is lower for the scenarios where the source plane spans a larger spatial area, thus a higher amount of residual dead times is present. Furthermore, the addition of a reflection in R2-C1 clearly manifests in a larger error, opposed to A2-C1. Somewhat surprisingly, the error decay is noticeably lower for A1-C1 compared to A2-C1 when, at first sight, both scenarios should be fairly similar after rectification of the dead times. This phenomenon cam be explained by the directivity of the employed measurement loudspeaker \cite[Fig. 5]{kujawski2024}. For the closer source-receiver-plane distance, a wider range of radiation angles is captured by the measurements, leading to a higher complexity of dynamics that need to be encoded in the \ac{rom}. This effect of source directivity is negligible in the \ac{mird} scenarios because, here, the source speakers are always pointed directly at the receivers. Because the geometric arrangement is not altered across the scenarios, the different error decays in \cref{fig:overview-mird} are explained entirely by the different reverberation times of the scenarios.

Theoretically, the computational runtime of randomized \ac{era} scales linearithmically (“n-log-n”) with the length of the \acp{ir} and quadratically with the order \(r\) of the constructed \ac{rom} \cite[Table 3.1]{minster2021}. For the considered benchmarks, the runtimes indeed are in \(\bigO{r^2}\) as can be gathered from \cref{fig:runtime}. As the reported runtimes not only include the adaptive \ac{era} routine but also the perpetual a posteriori evaluation of the relative error, this speaks towards the efficiency of our implementation.
\begin{figure}
\centering
\begin{subfigure}{.5\textwidth}
    %\captionsetup{aboveskip=-1em}
    \centering
    \includegraphics{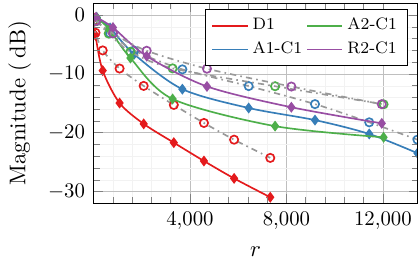}        
    \caption{\acs{miracle} scenarios.\label{fig:overview-miracle}}
\end{subfigure}%
\begin{subfigure}{.5\textwidth}
    %\captionsetup{aboveskip=-1em}
    \hfill
    \includegraphics{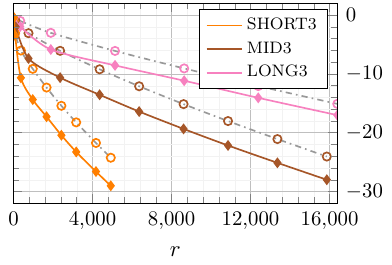}
    \caption{\acs{mird} scenarios.\label{fig:overview-mird}}
\end{subfigure}
\caption{Performance overview of randomized \ac{era} and the proposed \ac{loo} error estimator for all considered benchmarks. The plot depicts the relative error (diamonds) and estimator (circles) in dB over model order. Diamonds indicate orders for which a model was constructed.}
\end{figure}
Finally, a qualitative impression of the \acp{rom} is given in \cref{fig:d1-ir,fig:d1-tf,fig:largest-tf} where single channel \acp{ir} and frequency responses are depicted alongside the measurement data and resulting modelling error.
\begin{figure}
\centering
\begin{subfigure}{.5\textwidth}
    %\captionsetup{aboveskip=-1em}
    \centering
    \includegraphics{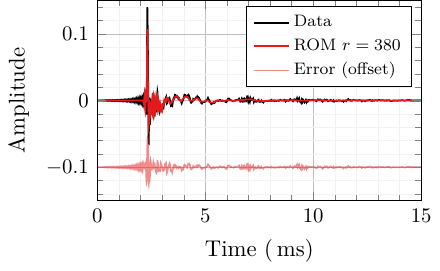}
    \caption{\ac{rom} of order \(380\).}
\end{subfigure}%
\begin{subfigure}{.5\textwidth}
    %\captionsetup{aboveskip=-1em}
    \hfill
    \includegraphics{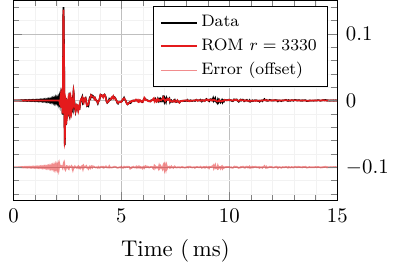}
    \caption{\ac{rom} of order \(3330\).}
\end{subfigure}
\caption{\acp{ir} of different \acp{rom} for the \acs{miracle}-D1 scenario. Each subplot depicts the \ac{ir} of the \ac{rom} alongside the \ac{ir} measurement and the resulting modelling error which is offset by \(-0.1\) for better visibility. Each \ac{rom} is actually a \ac{mimo} system with \(1024\) inputs and \(64\) outputs; the subplots only contain the transmission between the first input and first output.\label{fig:d1-ir}}
\end{figure}

\begin{figure}
\centering
\begin{subfigure}{.5\textwidth}
    %\captionsetup{aboveskip=-1em}
    \centering
    \includegraphics{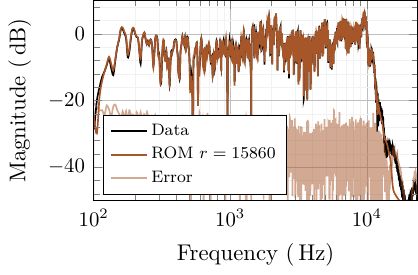}
    \caption{Largest \ac{rom} for MID3 scenario.\label{fig:mid3-tf}}
\end{subfigure}%
\begin{subfigure}{.5\textwidth}
    %\captionsetup{aboveskip=-1em}
    \hfill
    \includegraphics{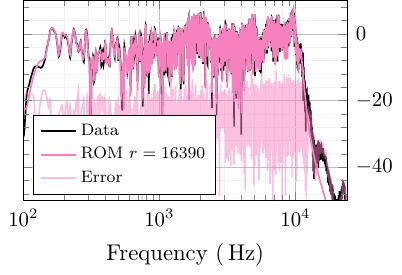}
    \caption{Largest \ac{rom} for LONG3 scenario.\label{fig:long3-tf}}
\end{subfigure}
\caption{Frequency responses of the largest constructed \acp{rom}. Each subplot depicts the frequency response of the \ac{rom} alongside the magnitude of the Fourier-transformed \ac{ir} measurement and the resulting modelling error. Each \ac{rom} is actually a \ac{mimo} system with \(26\) inputs and \(8\) outputs; the subplots only contain the transmission between the first input and first output.\label{fig:largest-tf}}
\end{figure}
\begin{figure}
\centering
\begin{subfigure}{.5\textwidth}
    %\captionsetup{aboveskip=-1em}
    \centering
    \includegraphics{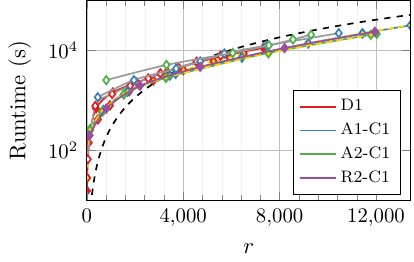}
    \caption{\ac{miracle} scenarios.}
\end{subfigure}%
\begin{subfigure}{.5\textwidth}
    %\captionsetup{aboveskip=-1em}
    \hfill
    \includegraphics{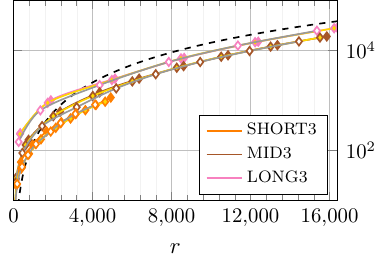}
    \caption{\ac{mird} scenarios.}
\end{subfigure}
\caption{Overview of the computational runtimes.\label{fig:runtime}. The plot depicts the runtime in seconds over model order. The colours encode the scenario, whereas the line style encodes the employed dead time extraction scheme according to \cref{fig:dtc}. For reference, the graph of a suitably scaled quadratic function is depicted as a black dashed line.}
\end{figure}
\section{Conclusion}
\label{sec:conclusion}
This paper presents significant advancements in the field of reduced order modelling for discrete-time \ac{lti} systems with input-output dead time based on \ac{ir} measurement data, particularly in acoustical engineering applications. We introduce a novel technique for extracting input and output dead times from measurement data by formulating and solving a linear program, which we refer to as the \acl{dts} problem. This approach enables the construction of more accurate, therefore efficient, \acp{rom} by better compensating for a wider range of propagation delays compared to previous methods.
We further develop an adaptive, randomized \ac{era} pipeline that leverages recent advances in numerical linear algebra, such as a randomized \ac{loo} error estimator and a memory-efficient QR algorithm. This enables scalable model reduction for large-scale datasets. Our software implementation is made available as open-source Python code \cite{pelling2025}, promoting reproducibility and enabling engineers and further research.
The proposed method is benchmarked on high-dimensional datasets, including \ac{miracle} \cite{kujawski2023} and \ac{mird} \cite{hadad2014}, representing the largest application of \ac{era} in terms of input data dimension and constructed model size.
Further, this paper identifies and corrects a typing error in the classical \ac{era} error bound \cite{kung1978}, providing a more accurate theoretical foundation for error estimation.

The proposed method can be extended to other fields where large-scale, data-driven \ac{mor} is required, such as control systems, signal processing, and multi-physics simulations.
The computational efficiency and robustness of the pipeline could be improved further, possibly by integrating more advanced randomized algorithms or out-of-core computations, to enable the construction of larger \acp{rom}. 
Recent development of tensor-based approximation methods offer ways to further exploit structure in the measurement data. Both a comparison to direct low-rank tensor modelling methods \cite{jalmby2021} but also the incorporation of low-rank tensor approximations directly into the \ac{era} pipeline as done in \cite{kilmer2022} could be avenues of future research. A \ac{dts}-based dead time extraction might also prove beneficial in both approaches.
Even though we show that the proposed randomized error estimator performs well in our applications, it is only heuristic and lacks a thorough mathematical analysis. Recent work \cite{oymak2022,he2025} demonstrates that the classical results on error analysis for \ac{era} are worth revisiting and offer new approaches.
\section*{Funding sources}
Funded by the Deutsche Forschungsgemeinschaft (DFG, German
Research Foundation) – project no.: \href{https://gepris.dfg.de/gepris/projekt/504367810?language=en}{504367810}.

\section*{Acknowledgements}
The authors would like to thank Ethan Epperly for his insightful comments on randomized error estimation, as well as providing a detailed mathematical derivation of the leave-one-out error estimator with power iterations.

\section{Acronyms}
\begin{acronym}[MIRACLE]\itemsep0pt
    \acro{blas}[BLAS]{Basic Linear Algebra Subprograms}
    \acro{bt}[BT]{Balanced Truncation}
    \acro{dft}[DFT]{discrete Fourier transform}
    \acro{dts}[DTS]{dead time splitting}
    \acro{dof}[DOF]{degree of freedom}
    \acroplural{dof}[DOFs]{degrees of freedom}
    \acro{era}[ERA]{Eigensystem Realization Algorithm}
    \acro{evd}[EVD]{eigenvalue decomposition}
    \acro{fir}[FIR]{finite impulse response}
    \acro{flop}[FLOP]{floating point operation}
    \acro{fom}[FOM]{full order model}
    \acro{hrtf}[HRTF]{head-related transfer function}
    \acro{iir}[IIR]{infinite impulse response}
    \acro{ir}[IR]{impulse response}
    \acro{io}[IO]{input-output}
    \acro{lapack}[LAPACK]{Linear Algebra Package}
    \acro{lti}[LTI]{linear time-invariant}
    \acro{loo}[LOO]{leave-one-out}
    \acro{mimo}[MIMO]{multiple-input-multiple-output}
    \acro{miracle}[MIRACLE]{\textit{Microphone Array Impulse Response Database for Acoustic Learning}}
    \acro{mird}[MIRD]{\textit{Multi-Channel Impulse Response Database}}
    \acro{miso}[MISO]{multiple-input-single-output}
    \acro{mor}[MOR]{model order reduction}
    \acro{rir}[RIR]{room impulse response}    
    \acro{rom}[ROM]{reduced order model}
    \acro{rsvd}[RSVD]{randomized \acs{svd}}
    \acro{simo}[SIMO]{single-input-multiple-output}    
    \acro{siso}[SISO]{single-input-single-output}
    \acro{svd}[SVD]{singular value decomposition}
    \acro{tde}[TDE]{time delay estimation}
    \acro{tera}[TERA]{Tangential Eigensystem Realization Algorithm}
\end{acronym}
\appendix
\section{Numerical Algorithms}
\label{sec:algos}
For the sake of reproducibility, this section contains a brief description of all numerical subroutines that are required by \cref{algo:randera}.
The computational centrepiece of randomized \ac{era} is \ac{rsvd}. We will not provide a derivation of the \ac{rsvd} algorithm here. Instead, the reader is referred to \cite{halko2011} for details. Many variants and augmentations of \ac{rsvd}, or rather the so-called \emph{randomized range finder}, can be found in the literature \cite{martinsson2020,halko2011}. Since the intended use is an adaptive pipeline that will draw new samples as necessary, we choose a formulation without oversampling. The following augmentations are made to the core \ac{rsvd} algorithm (\cite[Alg. 5.1]{halko2011}):
\begin{description}
    \item[Adaptivity] The \ac{loo} error estimator proposed in \cite{epperly2024} is used as a quality measure of the decomposition. If the estimator lies above the prescribed tolerance \(\etol\), the range basis is refined by drawing additional random samples (line \ref{alg:line:O}). Importantly, this refinement is done efficiently by updating the existing basis (line \ref{alg:line:update}). To see how the error tolerance \(\etol\) of the error estimator relates to the \(\htwo\)-error of the \ac{rom}, see \cref{sec:err-est}.
    \item[Power iterations] In previous numerical experiments, we observed that power iterations (or subspace iterations) greatly improve the accuracy of \ac{rsvd} in our applications \cite{pelling2021,kujawski2024}. Hence, \cref{algo:randsvd} includes power iterations (line \ref{alg:line:Y}). The number of power iterations in the reported experiments was fixed to \(q=2\) as it offered a substantial improvement of the error decay. Any higher number of power iterations leads to diminishing returns in terms of error. The use of power iterations also affects the \ac{loo} error estimator. Hence, a version with power iterations was used, which is detailed in \cref{algo:looest}.
    \item[Communication-avoiding orthogonalization] During numerical tests, it became apparent that the perpetual orthogonalization of the random samples is heavily constrained by memory bandwidth for larger model orders, rendering the default (modified) Gram-Schmidt algorithm in pyMOR ineffective. Since access to individual matrix elements is not possible in pyMOR's abstract vector array interface, a so-called communication-avoiding \emph{triangular orthogonalization} algorithm based on the shifted CholeskyQR algorithm \cite{fukaya2020} was implemented instead. To enable our proposed adaptive pipeline, an updating scheme was developed based on \cref{prop:cholqrupdate}. The orthogonalization step with the proposed procedure was up to \(300\) times faster than Gram-Schmidt for the larger benchmarks.
\end{description}
\begin{myalgo}
\caption{(\textsc{RandSVD}): Adaptive \acs{rsvd} with \acs{loo} error estimator.\label{algo:randsvd}}
    \begin{algorithmic}[1]
        \label[randsvd]{randsvd}
        \Require {Matrix \(X\in\R[m][n]\), block size \(q\in\N\), number of power iterations \(q\in\N_0\), approximation tolerance \(\etol>0\), machine epsilon \(\meps>0\).}
        \Ensure {\(U\in\R[m][r]\), \(\bSigma\in\R[r][r]\), \(V\in\R[n][r]\) satisfying \(\mathbb{E}(\norm[\mathrm{F}]{\smash{U\bSigma V^\top-X}})\leq\etol\).}
        \State \(\bOmega\in\R[n][b]\gets\texttt{randn}(n,b)\) \Comment{Standard Gaussian distribution.}
        \State \(Z\gets X\bOmega\)
        \State \(Y\gets(XX^\top)^qZ\) \label{alg:line:Y}
        \State \(Q,R\gets\textsc{ShiftedCholQR}(Y,\meps)\) \Comment{\cref{algo:cholqr}}
        \While {\(\textsc{LOOEst}(Z, Q, R, q)>\etol\)} \Comment{\cref{algo:looest}}
            \State \(\bOmega\in\R[n][b]\gets\texttt{randn}(n,b)\) \Comment{Standard Gaussian distribution.} \label{alg:line:O}
            \State \(Z_b\gets X\Omega\)
            \State \(Y_b\gets (XX^\top)^qZ_b\)
            \State \(Q,R\gets\textsc{CholQRUpdate}(Q,R,Y_b,\meps)\) \label{alg:line:update} \Comment{\cref{algo:cholqrupdate}}
            \State \(Z\gets [Z\quad Z_b]\)
        \EndWhile
        \State \(\tilde{U},\bSigma,V\gets\texttt{svd}(Q^\top X)\) \Comment{Economy \ac{svd}}
        \State \(U\gets Q\tilde{U}\)
    \end{algorithmic}
\end{myalgo}
\begin{myalgo}
\caption{(\textsc{LOOest}): \ac{loo} error estimator with power iterations \cite[Supplementary Materials: Program SM5]{epperly2024}, \cite{epperly2024pc}. \label{algo:looest}}
    \begin{algorithmic}[1]
        \label[looest]{looest}
        \Require{Random range samples \(Z\in\R[m][r]\), orthonormal  basis \(Q\in\R[m][r]\), upper-triangular factor \(R\in\R[r][r]\), and number of power iterations \(q\in\N_0\) as in \cref{algo:randsvd}.}
        \Ensure{\ac{loo} error estimator \(\eloo\) satisfying \cref{eq:eloo-expec}.}
        \State \(\begin{bmatrix}e_1&\cdots&e_r\end{bmatrix}\gets R^{-\top}\)        
        \If {\(q=0\)}
            \State \(\eloo\gets \sqrt{r^{-1}\sum_{j=1}^r\norm{e_j}^{-2}}\)
        \Else
            \State \(\mathbf{T}=\begin{bmatrix}t_1&\cdots&t_r\end{bmatrix}\gets \begin{bmatrix}e_1/\norm{e_1}&\cdots&e_r/\norm{e_r}\end{bmatrix}\)
            \State \(d\gets\begin{bmatrix}t_1Qz_1&\cdots&t_rQz_r\end{bmatrix}\)
            \State \(\eloo\gets r^{-1/2}\norm[\mathrm{F}]{\smash{Z-QQ^\top Z+Q\mathbf{T}\cdot\diag{d}}}\)
        \EndIf
    \end{algorithmic}
\end{myalgo}
\begin{algorithm}
\caption{(\textsc{ShiftedCholQR}): Iterated Cholesky QR with shift recomputation \cite{fukaya2020}.\label{algo:cholqr}
}
    \begin{algorithmic}[1]
        \label[cholqr]{cholqr}
        \Require{Matrix \(Y\in\R[m][r]\), machine epsilon \(\meps>0\).}
        \Ensure{\(Q\in\R[m][r]\), \(R\in\R[r][r]\) satisfying \(QR=Y\).}
        \State \(Q\gets Y\)
        \State \(R\gets I_r\)        
        \State \(X\gets Y^\top Y\)
        \While {\(\norm[\mathrm{F}]{\smash{X-I_r}}\geq\meps\sqrt{r}\)}
            \State {\(\tilde{R}\gets\texttt{chol}(X)\)} \Comment{Cholesky decomposition.}
            \If {\(\texttt{chol}(X)\) breaks down}
                \State \(\sigma\gets 11(mr+r(r+1))\meps\norm{X}\) \Comment{Recompute shift.}
                \State \(X\gets X+\sigma I_r\)
                \State \(\tilde{R}\gets\texttt{chol}(X)\) \Comment{Cholesky decomposition.}
            \EndIf
        \State \(Q\gets Q\tilde{R}^{-1}\)
        \State \(R\gets\tilde{R}R\)
        \State {\(X\gets Q^\top Q\)}        
        \EndWhile
    \end{algorithmic}
\end{algorithm}
\begin{myalgo}
\caption{(\textsc{CholQRUpdate}): Column update for shifted CholeskyQR.\label{algo:cholqrupdate}}
    \begin{algorithmic}[1]
        \label[cholqrupdate]{cholqrupdate}
        \Require{Existing QR decomposition \(Q\in\R[m][r]\), \(R\in\R[r][r]\), new columns \(Y_b\in\R[m][b]\), machine epsilon \(\meps>0\).}
        \Ensure{\(\tilde{Q}\in\R[m][r]\), \(\tilde{R}\in\R[r][r]\) satisfying \(\tilde{Q}\tilde{R}=[QR\quad Y_b]\).}
        \State \(Q_b\gets Y_b\)
        \State \(R_b\gets I_b\)
        \State \(B\gets Q^\top Y_b\)
        \State \(X\gets Y_b^\top Y_b-BB^\top\)
        \While {\(\norm[\mathrm{F}]{\smash{Q_b^\top Q_b-I_b}}\geq\meps\sqrt{b}\)}
            \While {\(\tilde{R}_b\gets\texttt{chol}(X)\) breaks down}
                \State \(\sigma\gets 11(mr+r(r+1))\meps\norm{X}\) \Comment{Recompute shift.}
                \State \(X\gets X+\sigma I_b\)
            \EndWhile
        \State \(Q_b\gets (Q_b-QB)\tilde{R}_b^{-1}\)
        \State \(R_b\gets\tilde{R}_bR_b\)                
        \State \(B\gets Q^\top Q_b\) \Comment{Update \(B\) for next iteration.}        
        \State \(X\gets Q_b^\top Q_b-BB^\top\) \Comment{Update \(X\) for next iteration.}
        \EndWhile        
        \State \(\tilde{Q}\gets\begin{bmatrix}
            Q&Q_b
        \end{bmatrix}\)
        \State \(\tilde{R}\gets \begin{bmatrix}
            R&B\\0&R_b
        \end{bmatrix}\)
    \end{algorithmic}
\end{myalgo}
\clearpage
\begin{proposition}
\label{prop:cholqrupdate}
    Let \(Q\in\R[m][r]\), \(R\in\R[r][r]\) be a QR decomposition \(QR=Y\in\R[m][r]\) and let \(R_b\) be given by the Cholesky factorization
    \begin{equation*}
         R_b^\top R_b=(Y_b -QQ^\top Y_b)^\top(Y_b-QQ^\top Y_b)=Y_b^\top Y_b-Y_b^\top QQ^\top Y_b
    \end{equation*} for some \(Y_b\in\R[m][b]\). Then,
    \begin{equation*}
        \tilde{Q}=\begin{bmatrix}
            Q&(Y_b-QQ^\top Y_b)R_b^{-1}
        \end{bmatrix}\quad\mathrm{and}\quad\tilde{R}=
        \begin{bmatrix}
            R&Q^\top Y_b\\0&R_b
        \end{bmatrix}
    \end{equation*}
    is a QR decomposition of \(\tilde{Y}=[Y\quad Y_b]\in\R[m][(r+b)]\).
\end{proposition}
\begin{proof}
Firstly, we have
\begin{equation*}
    \tilde{Q}\tilde{R}=\begin{bmatrix}
        QR&QQ^\top Y_b+ (Y_b -QQ^\top Y_b)R_b^{-1}R_b
    \end{bmatrix}=\begin{bmatrix}
        Y&Y_b
    \end{bmatrix}=\tilde{Y}.
\end{equation*}
Secondly, \(\tilde{R}\) is upper triangular and \(\tilde{Q}\) is orthogonal as can be verified by
\begin{align*}
    \tilde{Q}^\top\tilde{Q}\!&=\!\!\begin{bmatrix}
        Q^\top Q&Q^\top(Y_b-QQ^\top Y_b)R_b^{-1}\\R_b^{-\top}\!(Y_b^\top-Y_b^\top QQ^\top)Q&R_b^{-\top}\!(Y_b- QQ^\top Y_b)^\top\!(Y_b-QQ^\top Y_b)R_b^{-1}
        \end{bmatrix}\\
    &=I_{r+b}.
\end{align*}
\end{proof}
\bibliographystyle{elsarticle-num}
\bibliography{references.bib}
\end{document}